\documentclass[a4paper,11pt]{article}
\usepackage[T1]{fontenc}
\usepackage{babel}

\usepackage{hyperref}

\usepackage{float}

\usepackage{amssymb,amscd}
\usepackage[centertags]{amsmath}
\usepackage{amsfonts}
\usepackage{enumerate}
\usepackage{latexsym}
\usepackage{amsthm}
\usepackage{amssymb}
\usepackage{multicol}
\usepackage[table]{xcolor}

\usepackage{mathtools}
\usepackage{cancel}

\usepackage{tikz}
\usetikzlibrary{fit}
\usetikzlibrary{shapes.geometric}
\usetikzlibrary{graphs}
\usetikzlibrary{positioning}
\usetikzlibrary{calc}
\usepackage{ifthen}
\usepackage[normalem]{ulem}

\usepackage{comment}

\usepackage{hyperref}

\usepackage[numbers,sort]{natbib} 

\usepackage{soul}

\usepackage{authblk}

\usepackage{amssymb}
\usepackage{amsthm}
\usepackage{amsmath}
\usepackage{color}
\usepackage{tikz}
\usepackage{hyperref}

    \newcommand{\Kneser}[2]{K(#1, #2)}
    \newcommand{\Johnson}[2]{J(#1, #2)}
    
    \newcommand{\dom}[0]{\gamma}
    \newcommand{\tdom}[0]{\gamma_t}
    \newcommand{\N}[0]{\mathbb{N}}

    \newtheorem{theorem}{Theorem}
    \newtheorem{coro}[theorem]{Corollary}
    \newtheorem{lemma}[theorem]{Lemma}
    \newtheorem{prop}[theorem]{Proposition}
    \newtheorem{remark}[theorem]{Remark}
    \newtheorem{problem}[theorem]{Problem}
    \newtheorem{ex}[theorem]{Example}

\begin{document}

\title{Domination and Total Domination in Johnson graphs\footnote{MGC and PT were partially supported by PIP CONICET 1900, PICT-2020-03032, PPCT-UNR 80020250400056UR and PID 80020210300068UR. TD acknowledge the financial support of the Slovenian Research and Innovation Agency (research core funding No.\ P1-0297 and project N1-0431).}}

\author[1,2]{Maria Gracia Cornet}
\author[3,4]{Tanja Dravec}
\author[1,2]{Pablo Torres}
\date{}

\affil[1]{\footnotesize Depto. de Matem\'atica, Facultad de Ciencias Exactas, Ing. y Agrimensura, Universidad Nacional de Rosario}
\affil[2]{\footnotesize Consejo Nacional de Investigaciones Científicas y Técnicas, Argentina}
\affil[3]{\footnotesize University of Maribor, Slovenia}
\affil[4]{\footnotesize Institute of Mathematics, Physics and Mechanics, Ljubljana, Slovenia}

\maketitle

\begin{abstract}
    In this paper we study domination and total domination of Johnson graphs $\Johnson{n}{k}$. 
    We establish monotonicity results proving that both $\dom(\Johnson{n}{k})$ and $\tdom(\Johnson{n}{k})$ are non-decreasing with respect to $n$.
    We compute exact values for $\tdom(\Johnson{n}{k})$ in specific cases, showing that $\tdom(\Johnson{n}{2})=\lceil\frac{2}{3}(n-1)\rceil$ for $n\geq 4$, and we determine $\tdom(\Johnson{n}{3})$ for $n\geq 6$, which depends quadratically on $n$.

\end{abstract}

\section{Introduction}
\label{sec:introduction}

Given a graph $G$, $D\subseteq V(G)$ is a \textit{dominating set} of $G$ if for every $v\in V(G)\setminus D$ there exists $u \in D$ such that $uv \in E(G)$. The \textit{domination number} of $G$, denoted by $\dom(G)$, is the least cardinality among all the dominating sets of $G$. A \textit{$\dom$-set} of $G$ is a dominating set of $G$ with cardinality $\dom(G)$.
Domination is one of the classical concepts in graph theory. It was introduced over 60 years ago and has undergone extensive development. The topic was comprehensively surveyed in recent monographs by Haynes, Hedetniemi, and Henning~\cite{HHH-2021,HHH-2023}. For graph theory notation and terminology, we refer to~\cite{HHH-2023}.

Total domination, introduced by Cockayne, Dawes and Hedetniemi
\cite{CDH1980}, strengthens classical domination by requiring that \emph{every} vertex of the graph --- including those in the dominating set itself --- have a neighbor in the set. Formally, a set $D\subseteq V(G)$ is a \emph{total dominating set} of a graph $G$ with no isolated vertices if $N(S)=V(G)$, and the \emph{total domination number}, $\tdom(G)$, is the minimum cardinality of such a set. The parameter has been studied extensively since its introduction; we refer to the monograph of Henning and Yeo \cite{henningyeo-13} and to Henning's earlier survey \cite{Henning2009}. Both domination and total domination are NP-hard to compute, and their corresponding decision problems are NP-complete~\cite{garey1990guide,PfaffLaskarHedetniemi1983}, which motivates the search for closed-form or tight asymptotic values of these parameters on structured graph families.


Among such families, distance-regular graphs arising from set-theoretic constructions have received particular attention. 
For $k,n\in\N$ with $k\leq n$, let $[k..n]$ and $[n]$ denote the sets $\{k,\ldots,n\}$ and $\{1,\ldots,n\}$, respectively. Given a set $S$, we denote by $\binom{S}{k}$ the collection of all the subsets of $S$ with cardinality $k$.
For $n>k\ge1$, the \emph{Johnson graph} $\Johnson{n}{k}$ has vertex set $\binom{[n]}{k}$, with two $k$-subsets adjacent whenever they intersect in exactly $k-1$ elements; it is $k(n-k)$-regular of order $\binom{n}{k}$. Johnson graphs arise naturally in the study of constant-weight codes, covering and packing designs, and Erd\H{o}s--Ko--Rado type extremal set theory. Besides, they are among the most studied families of distance-regular graphs. Regarding Johnson graphs, various problems and parameters have been studied, e.g.\ stability and chromatic numbers~\cite{BSSS-2006,EB-1996}, competition numbers~\cite{KPS-2010}, automorphism group~\cite{RD-2011},
distances and diameter~\cite{AACHT-2018}, and clique covering numbers~\cite{J-2025}.

\begin{figure}[ht]
    \centering
    \begin{tikzpicture}[scale=0.8]
        \foreach \x in {0,1,...,9}{
        \draw ({90+\x*36}:1.5) to[bend right=10] ({90+(\x+1)*36}:1.5);
        \draw ({90+\x*36}:1.5) arc (54+(\x)*36:234+(\x)*36:{1.5*sin(54)});
        }
        \foreach \x in {0,1,...,4}{
        \draw ({90+\x*72}:1.5) to[bend left=10] ({90+(\x+1)*72}:1.5);
        \draw ({54+\x*72}:1.5) to[bend right=10] ({54+(\x+2)*72}:1.5);
        }

        \foreach \x in {0,1,...,9}{
        \begin{scope}[shift={({90+\x*36}:1.5)}]
            \filldraw[fill=white, thick] (0,0) circle (0.2);
            \foreach \y in {0,1,...,4}{
            \draw (0,0) -- ({90-\y*72}:0.2);
            }
        \end{scope}
        }

        \begin{scope}[shift={(90:1.5)}]
            \filldraw[fill=gray!70] (0,0) -- (90:0.2) arc (90:18:0.2) -- cycle;
            \filldraw[fill=gray!70] (0,0) -- (18:0.2) arc (18:-54:0.2) -- cycle;
        \end{scope}
        \begin{scope}[shift={(18:1.5)}]
            \filldraw[fill=gray!70] (0,0) -- (18:0.2) arc (18:-54:0.2) -- cycle;
            \filldraw[fill=gray!70] (0,0) -- (-54:0.2) arc (-54:-126:0.2) -- cycle;
        \end{scope}
        \begin{scope}[shift={(306:1.5)}]
            \filldraw[fill=gray!70] (0,0) -- (-54:0.2) arc (-54:-126:0.2) -- cycle;
            \filldraw[fill=gray!70] (0,0) -- (-126:0.2) arc (-126:-198:0.2) -- cycle;
        \end{scope}
        \begin{scope}[shift={(234:1.5)}]
            \filldraw[fill=gray!70] (0,0) -- (-126:0.2) arc (-126:-198:0.2) -- cycle;                \filldraw[fill=gray!70] (0,0) -- (-198:0.2) arc (-198:-270:0.2) -- cycle;
        \end{scope}
        \begin{scope}[shift={(162:1.5)}]
            \filldraw[fill=gray!70] (0,0) -- (90:0.2) arc (90:18:0.2) -- cycle;
            \filldraw[fill=gray!70] (0,0) -- (-198:0.2) arc (-198:-270:0.2) -- cycle;
        \end{scope}
        \begin{scope}[shift={(54:1.5)}]
            \filldraw[fill=gray!70] (0,0) -- (90:0.2) arc (90:18:0.2) -- cycle;
            \filldraw[fill=gray!70] (0,0) -- (-54:0.2) arc (-54:-126:0.2) -- cycle;
        \end{scope}
        \begin{scope}[shift={(342:1.5)}]
            \filldraw[fill=gray!70] (0,0) -- (18:0.2) arc (18:-54:0.2) -- cycle;
            \filldraw[fill=gray!70] (0,0) -- (-126:0.2) arc (-126:-198:0.2) -- cycle;
        \end{scope}
        \begin{scope}[shift={(270:1.5)}]
            \filldraw[fill=gray!70] (0,0) -- (-54:0.2) arc (-54:-126:0.2) -- cycle;
            \filldraw[fill=gray!70] (0,0) -- (-198:0.2) arc (-198:-270:0.2) -- cycle;
        \end{scope}
        \begin{scope}[shift={(198:1.5)}]
            \filldraw[fill=gray!70] (0,0) -- (90:0.2) arc (90:18:0.2) -- cycle;
            \filldraw[fill=gray!70] (0,0) -- (-126:0.2) arc (-126:-198:0.2) -- cycle;
        \end{scope}
        \begin{scope}[shift={(126:1.5)}]
            \filldraw[fill=gray!70] (0,0) -- (18:0.2) arc (18:-54:0.2) -- cycle;
            \filldraw[fill=gray!70] (0,0) -- (-198:0.2) arc (-198:-270:0.2) -- cycle;
        \end{scope}
    \end{tikzpicture}
    
    \caption{Johnson graph $\Johnson{5}{2}$}
    \label{fig: J(5,2)}
\end{figure}
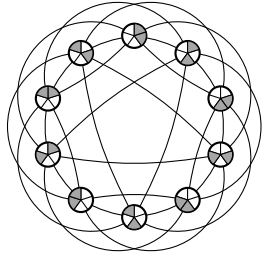

By contrast, domination-type parameters of Johnson graphs remained largely open until recently. Ivan\v{c}o and Zelinka \cite{IZ-1993}, Gorodezky \cite{G-2007} and \"{O}sterg\aa rd, Shao and Xu \cite{OSX-2014} studied domination in the closely related Kneser graphs $\Kneser{n}{k}$, obtaining general bounds but no closed form outside restricted ranges of $n$.
The classical domination number of $\Johnson{n}{k}$ was addressed only very recently: in \cite{CDT-2025} we studied $\dom(\Johnson{n}{k})$ in closed form for $k=2$ and, via a bijective correspondence with edge coverings by triangles of auxiliary graphs on $[n]$ built on the classical Fort--Hedlund covering numbers \cite{FortHedlund1958}, for $k=3$ and even $n\ge6$; we conjectured the same closed form, $\dom(\Johnson{n}{3})=\phi_n$, for odd $n$. This conjecture was recently settled by Lee and Oh \cite{lee2026}, completing the determination of $\dom(\Johnson{n}{3})$ for all $n\ge6$.
\begin{theorem}[\cite{CDT-2025,lee2026}]\label{th:k=3}
    If $n\geq 6$, then $\dom(\Johnson{n}{3})=\phi_n$, where
    $$\phi_n=\begin{cases}
        \frac{n^2-2n+c_n}{12}, & \text{if $n$ is even},\\
        \frac{n^2-n+c_n}{12}, & \text{if $n$ is odd},
    \end{cases}$$
    with
    \begin{equation}\label{eq: cn}
        c_n=\begin{cases}
        0, & \text{if }n\equiv 1,2,6 \mod 12,\\
        4, & \text{if }n\equiv 4,5 \mod 12,\\
        6, & \text{if }n\equiv 3,7 \mod 12,\\
        10, & \text{if }n\equiv 11 \mod 12,\\
        12, & \text{if }n\equiv 0,8,9 \mod 12,\\
        16, & \text{if }n\equiv 10 \mod 12,
    \end{cases}
    \end{equation}
\end{theorem}
As a corollary, the strict monotonicity $\dom(\Johnson{n}{k})<\dom(\Johnson{n+1}{k})$ that we had conjectured in \cite{CDT-2025} is confirmed for $k=3$. For $k\ge4$, determining $\dom(\Johnson{n}{k})$ remains open. 
A well-known upper bound on the domination number of a graph $G$ is its independence number $\alpha(G)$, which in the case of Johnson graph $\Johnson{n}{k}$, is at most $\binom{n}{k}/(n-k+1)$ (see~\cite{GM-2015}). Besides, recall that for any graph $G$ of order $n$ and maximum degree $\Delta(G)$, it holds that $\dom(G)\geq\frac{n}{\Delta(G)+1}$ (see, e.g. \cite{HHH-2023}). Therefore, it holds
\begin{equation}
\label{eq: bounds-J(n,k)}
     \frac{\binom{n}{k}}{k(n-k)+1}\leq \dom(\Johnson{n}{k})\leq \frac{\binom{n}{k}}{n-k+1}.
\end{equation}
Related variants of domination have also been studied on these families, such as Roman domination by Zec \cite{Z-2023}.

Independently of the graph-theoretic viewpoint, the classical domination number of $\Johnson{n}{k}$ admits an equivalent formulation in the language of \emph{lotto designs}. An $LD(n,k,p,t;b)$ lotto design is a collection of $b$ $k$-subsets (\emph{blocks}) of an $n$-set such that every $p$-subset intersects some block in at least $t$ points. We denote by $L(n,k,p,t)$ the minimum possible number of blocks in such a design. A dominating set of $\Johnson{n}{k}$ is precisely a collection of blocks such that every $k$-subset intersects at least one of them in at lest $k-1$ points. Thus, the domination number corresponds to the case $p=k$, $t=k-1$, yielding
\begin{equation}
    \label{eq:dom-lotto}
    \dom(\Johnson{n}{k}) = L(n,k,k,k-1).
\end{equation}
This connection, which we noticed after \cite{CDT-2025} was published, places $\dom(\Johnson{n}{3})$ in a much older line of work: a closed formula for $L(n,3,3,2)$, matching $\phi_n$, was obtained by Bate \cite{Bate1978} almost five decades ago. The case $k=2$ is likewise a Lotto design triviality: $L(n,2,2,1)=\lceil n/2\rceil$. For $k\ge4$, $L(n,k,k,k-1)$ does not appear in the Lotto design tables of \cite{LiVanRees2002,BateVanRees1998} beyond small sporadic cases, so \eqref{eq:dom-lotto} does not close the gap left open by \cite{CDT-2025,lee2026}, but it does mean that any progress on $\dom(\Johnson{n}{k})$ for $k\ge4$ translates directly into new Lotto design values, and vice versa.

Despite this progress, \emph{total} domination has received no dedicated treatment on Johnson graphs. The only bounds available prior to this paper are generic ones inherited from the theory of total domination in regular graphs: the trivial bound $\tdom(G)\ge n/\Delta(G)$, sharp only for disjoint unions of complete bipartite graphs; the general upper bound $\tdom(G)\le\big(\tfrac{1+\ln\delta(G)}{\delta(G)}\big)n$ of Henning and Yeo \cite{henningyeo-13}; and, for fixed small degree, the sharp constants $\Gamma_0(d)$ of \cite{ Archdeaconetal2004,ThomasseYeo2007,DorflingHenning2015}. None of these use the combinatorial structure of $\Johnson{n}{k}$, and --- since the degree $k(n-k)$ grows with $n$ for fixed $k$ --- the fixed-degree constants $\Gamma_0(d)$ are in any case only applicable to isolated small instances. A structurally better starting point comes from the standard sandwich $\dom(G)\leq\tdom(G)\leq 2\dom(G)$ (\cite{CDH1980}), which combined with $\dom(\Johnson{n}{3})=\phi_n$ gives $\phi_n\leq\tdom(\Johnson{n}{3})\le 2\phi_n$; this already improves asymptotically on the trivial bound $n/\Delta$, but leaves a multiplicative gap of essentially $2$ between the two sides.

The paper is organized as follows. 
In Section~\ref{sec: dom-Johnson} we study the domination problem in this family of graphs.
First, we prove monotonicity results with respect to $n$: $\dom(\Johnson{n}{k})$ is non-decreasing, and in fact strictly increasing every two steps, $\dom(\Johnson{n}{k})<\dom(\Johnson{n+2}{k})$, for $n\geq 2k$
(these results appeared in proceedings of LAGOS 2025 \cite{CDT-2025}).
We then develop a recursive technique for deriving upper bounds for general $k$, improving the bound (\ref{eq: bounds-J(n,k)}) for $k\in\{4,5,6\}$.

In Section~\ref{sec: tdom-Johnson} we focus on the total domination number of the graphs $\Johnson{n}{k}$, studying the cases $k=2$ and $k=3$. For both we obtain the exact value of $\tdom(\Johnson{n}{k})$: we prove that $\tdom(\Johnson{n}{2})=\left\lceil 2(n-1)/3\right\rceil$ for all $n\geq 4$, and that $\tdom(\Johnson{n}{3})=\varphi_n$ for all $n\geq 10$, with $\varphi_n$ again an explicit quadratic expression. Our approach in this last case is translating the total dominating sets into coverings by edge-overlapping triangles (CEOT), and relying on known results on edge coverings and packings by diamonds ($K_4-e$) in complete graphs.

\section{Domination in Johnson graphs}\label{sec: dom-Johnson}

Given graphs $G$ and $H$, a {\emph{homomorphism}} from $G$ to $H$ is a function $f: V(G) \to V(H)$ such that if $x,y \in V(G)$ are adjacent in $G$, then $f(x)f(y) \in E(H)$. A function $f:V(G) \to V(H)$ is a {\emph{weak homomorphism}} if adjacent vertices of $G$ are mapped either to adjacent or equal vertices of $H$. It was already observed that the image of a dominating set under a surjective homomorphism is again a dominating set~\cite{G-2007}. 
Clearly, the result holds also for weak homomorphisms.

\begin{remark}\label{r:weakHomom}
    If there exists a surjective (weak) homomorphism from a graph $G$ to a graph $H$, then $\dom(G) \geq \dom(H)$.
\end{remark}

\begin{lemma}\label{l:JohnsonHomomor}
    If $n \geq 2k$, then there exists a surjective weak homomorphism from $\Johnson{n+1}{k}$ to $\Johnson{n}{k}$.
\end{lemma}

\begin{proof}
    Define $f:V(\Johnson{n+1}{k}) \to V(\Johnson{n}{k})$ in the following way. For any $u \in V(\Johnson{n+1}{k})$, let 
    \[ 
f(u)= \left\{
\begin{array}{ll}
      u, & n+1 \notin u, \\
      (u\setminus \{n+1\}) \cup \max([n]\setminus u),  & n+1 \in u.\\
\end{array} 
\right. 
\]
Function $f$ is clearly surjective since for any $x \in V(\Johnson{n}{k}) \subseteq V(\Johnson{n+1}{k})$ it holds that $f(x)=x$. Moreover, let $x,y$ be two adjacent vertices of $\Johnson{n+1}{k}$. Then $|x \cap y|=k-1$ and hence either $|f(x) \cap f(y)|=k-1$ or $f(x)=f(y)$.
\end{proof}

As a direct consequence of Lemma~\ref{l:JohnsonHomomor} and Remark~\ref{r:weakHomom}, we have the following.

\begin{theorem}\label{thm:monotonicity}
    If $n\geq 2k$, then $\dom(\Johnson{n}{k})\leq \dom(\Johnson{n+1}{k})$.
\end{theorem}

Next, we show that the domination number of Johnson graphs is not the same for three consecutive values of $n$, i.e. $\dom(\Johnson{n}{k}) < \dom(\Johnson{n+2}{k})$. To this end, given a dominating set $D$ of $\Johnson{n}{k}$, for each $S\subseteq D$ we consider $h_D(S)=|\bigcap_{x\in S}x|$. Note that $h_D(S)$ is the number of elements in $[n]$ that belong to every vertex of $S$. Let $h_D=\max\{|S|:\ h_D(S)\geq2\}$. Observe that if $k\geq2$, $h_D\geq1$ for every dominating set $D$.

\begin{theorem}\label{thm:strict - 2steps}
    Let $k\geq2$ and $n \geq 2k$. Then, $\dom(\Johnson{n}{k})+1\leq\dom(\Johnson{n+2}{k})$. 
    Moreover, if $k\geq 5$ then $\dom(\Johnson{n}{k})+2\leq\dom(\Johnson{n+2}{k})$.
\end{theorem}

\begin{proof}
Let $D$ be a $\dom$-set of $\Johnson{n+2}{k}$ and $S\subseteq D$ such that $h_D=|S|$. Note that $h_D(S)\geq 2$. Without loss of generality assume that $\{n+1,n+2\}\subseteq\bigcap_{x\in S}x$. By maximality of $S$, every vertex $v$ in $D_S=D\setminus S$ satisfies $|v\cap\{n+1,n+2\}|\leq1$.

For any $v \in D_S$  define
    \[ 
v'= \left\{
\begin{array}{ll}
      v, &  \text{if }v \cap \{n+1,n+2\}=\emptyset, \\
      (v\setminus \{n+1\}) \cup \max([n]\setminus v),  & \text{if }n+1 \in v,\\
      (v\setminus \{n+2\}) \cup \max([n]\setminus v),  & \text{if }n+2 \in v.\\
\end{array} 
\right. 
\]
Then $D'=\{v': v \in D_S\}$ is a dominating set of $\Johnson{n}{k}$. Indeed, let $x \in V(\Johnson{n}{k})\setminus D'$. Then $x \in V(\Johnson{n+2}{k})$ and hence there exists $y \in D$ with $|x \cap y|=k-1$. Since $n+1,n+2 \notin x$, it follows that $\{n+1,n+2\} \nsubseteq y$ and $|x \cap y'|=k-1$ (note that $|x \cap y'|=k$ is not possible, since $x \notin D'$). Therefore, $xy' \in E(\Johnson{n}{k})$ and thus $D'$ is a dominating set of $\Johnson{n}{k}$.

Hence, $\dom(\Johnson{n}{k})\leq |D'|=|D|-|S|=\dom(\Johnson{n+2}{k})-h_D$. As $h_D\geq 1$ for every $n\geq 2k$, it follows $\dom(\Johnson{n}{k})+1\leq \dom(\Johnson{n+2}{k})$.

Finally, assume $k\geq 5$. Notice that there exist two vertices in $D$ which are at distance at most $3$. Otherwise, let $u\in D$, and $w$ a vertex in $\Johnson{n+2}{k}$ which is at distance $2$ from $u$. Since the distance between $u$ and any other vertex of $D$ is at least $4$, $w\notin D$ and has no neighbor in $D$, contradicting the fact that $D$ is a dominating set. Therefore, there exist $u,v\in D$ such that $\mathrm{dist}(u,v)=k-|u\cap v|\leq 3$, this is, $|u\cap v|\geq k-3 \geq 2$. It follows $h_D\geq 2$ and as a consequence $\dom(\Johnson{n}{k})+2\leq \dom(\Johnson{n+2}{k})$.
\end{proof}

\begin{remark}
    If $k=4$ and $\Johnson{n+2}{k}$ does not admit a $1$-perfect code, this is, a set of vertices whose closed neighborhoods form a partition of $V(\Johnson{n+2}{k})$, then every dominating set $D$ contains at least two vertices at distance at most $2$. Using an argument similar to the one at the end of the proof of Theorem \ref{thm:strict - 2steps}, we obtain $h_D\geq 2$, and hence $\dom(\Johnson{n}{k})+2\leq\dom(\Johnson{n+2}{k})$.
\end{remark}

\subsection{Construction of dominating sets for \texorpdfstring{$\Johnson{n}{k}$}{J(n,k)}}\label{sec:bounds}

In this subsection we give a recursive construction of dominating sets that we will use to give upper bounds for $\dom(\Johnson{n}{k})$. The construction we present in Remark \ref{rem:recursive-bound} is inspired by the recursive construction proposed by Jorgensen for the study of the clique covering number of Johnson graphs~\cite{J-2025}.

Consider $D_1$ and $D_2$ dominating sets of $\Johnson{n-1}{k}$ and $\Johnson{n-1}{k-1}$, respectively. Define $D=D_1\cup \{w\cup\{n\}:w\in D_2\}$ and let $u$ be a vertex of $\Johnson{n}{k}$.
If $n\in u$, let $\hat{u}=u\setminus \{n\}$. Since $D_2$ is a dominating set of $\Johnson{n-1}{k-1}$, there exists $w\in D_2$ such that $|w\cap \hat{u}|\geq k-2$. Then $w\cup\{n\}\in D$ and $|(w\cup\{n\})\cap u|\geq k-1$.
Otherwise, $u\subseteq [n-1]$. Since $D_1$ is a dominating set of $\Johnson{n-1}{k}$, there exists $v\in D_1\subseteq D$ such that $|v\cap u|\geq k-1$.
Therefore, $D$ is a dominating set of $\Johnson{n}{k}$.

\begin{remark}\label{rem:recursive-bound}
If $D_1$ and $D_2$ are dominating sets of $\Johnson{n-1}{k}$ and $\Johnson{n-1}{k-1}$ respectively, then the set
$$D=D_1\cup \{w\cup\{n\}:w\in D_2\}$$
is a dominating set of $\Johnson{n}{k}$ and
$$\dom(\Johnson{n}{k})\leq\dom(\Johnson{n-1}{k-1})+\dom(\Johnson{n-1}{k}).$$
\end{remark}


\begin{ex}
    The set $D_1$ represented in Figure $\ref{fig:domset-J(8,4)}$ is a dominating set of the graph $\Johnson{8}{4}$. On the other hand, the set $D_2$ represented in Figure $\ref{fig:domset-J(8,3)}$ is a (minimum) dominating set of $\Johnson{8}{3}$.
    \begin{figure}[ht]
        \centering
        \begin{minipage}[b]{0.48\linewidth}
            \centering
            \begin{tabular}{|c|c|c|c|c|c|c|c|}
                \hline
                 1 & 2 & 3 & 4 & 5 & 6 & 7 & 8 \\ \hline
                \cellcolor[HTML]{DAE8FC}x & \cellcolor[HTML]{DAE8FC} & \cellcolor[HTML]{DAE8FC}x & \cellcolor[HTML]{DAE8FC}x & \cellcolor[HTML]{DAE8FC}x & \cellcolor[HTML]{DAE8FC} & \cellcolor[HTML]{DAE8FC} & \cellcolor[HTML]{DAE8FC}\\ \hline
                \cellcolor[HTML]{DAE8FC}x & \cellcolor[HTML]{DAE8FC} & \cellcolor[HTML]{DAE8FC}x & \cellcolor[HTML]{DAE8FC}x & \cellcolor[HTML]{DAE8FC} & \cellcolor[HTML]{DAE8FC}x & \cellcolor[HTML]{DAE8FC} & \cellcolor[HTML]{DAE8FC}\\ \hline
                \cellcolor[HTML]{DAE8FC}x & \cellcolor[HTML]{DAE8FC} & \cellcolor[HTML]{DAE8FC} & \cellcolor[HTML]{DAE8FC} & \cellcolor[HTML]{DAE8FC}x & \cellcolor[HTML]{DAE8FC}x & \cellcolor[HTML]{DAE8FC}x & \cellcolor[HTML]{DAE8FC}\\ \hline
                \cellcolor[HTML]{DAE8FC} & \cellcolor[HTML]{DAE8FC}x & \cellcolor[HTML]{DAE8FC} & \cellcolor[HTML]{DAE8FC} & \cellcolor[HTML]{DAE8FC}x & \cellcolor[HTML]{DAE8FC}x & \cellcolor[HTML]{DAE8FC} & \cellcolor[HTML]{DAE8FC}x\\ \hline
                \cellcolor[HTML]{DAE8FC} & \cellcolor[HTML]{DAE8FC}x & \cellcolor[HTML]{DAE8FC}x & \cellcolor[HTML]{DAE8FC} & \cellcolor[HTML]{DAE8FC} & \cellcolor[HTML]{DAE8FC} & \cellcolor[HTML]{DAE8FC}x & \cellcolor[HTML]{DAE8FC}x\\ \hline
                \cellcolor[HTML]{DAE8FC} & \cellcolor[HTML]{DAE8FC}x & \cellcolor[HTML]{DAE8FC} & \cellcolor[HTML]{DAE8FC}x & \cellcolor[HTML]{DAE8FC} & \cellcolor[HTML]{DAE8FC} & \cellcolor[HTML]{DAE8FC}x & \cellcolor[HTML]{DAE8FC}x\\ \hline
            \end{tabular}
            \caption{Dominating set of $\Johnson{8}{4}$.}
            \label{fig:domset-J(8,4)}
        \end{minipage}
        \hfill 
        \begin{minipage}[b]{0.48\linewidth}
            \centering
            \begin{tabular}{|c|c|c|c|c|c|c|c|}
                \hline
                 1 & 2 & 3 & 4 & 5 & 6 & 7 & 8\\ \hline
                \cellcolor[HTML]{ACE1AF}x & \cellcolor[HTML]{ACE1AF}x & \cellcolor[HTML]{ACE1AF}x & \cellcolor[HTML]{ACE1AF} & \cellcolor[HTML]{ACE1AF} & \cellcolor[HTML]{ACE1AF} & \cellcolor[HTML]{ACE1AF} & \cellcolor[HTML]{ACE1AF}\\ \hline
                \cellcolor[HTML]{ACE1AF}x & \cellcolor[HTML]{ACE1AF}x & \cellcolor[HTML]{ACE1AF} & \cellcolor[HTML]{ACE1AF}x & \cellcolor[HTML]{ACE1AF} & \cellcolor[HTML]{ACE1AF} & \cellcolor[HTML]{ACE1AF} & \cellcolor[HTML]{ACE1AF}\\ \hline
                \cellcolor[HTML]{ACE1AF}x & \cellcolor[HTML]{ACE1AF}x & \cellcolor[HTML]{ACE1AF} & \cellcolor[HTML]{ACE1AF} & \cellcolor[HTML]{ACE1AF}x & \cellcolor[HTML]{ACE1AF} & \cellcolor[HTML]{ACE1AF} & \cellcolor[HTML]{ACE1AF}\\ \hline
                \cellcolor[HTML]{ACE1AF} & \cellcolor[HTML]{ACE1AF} & \cellcolor[HTML]{ACE1AF}x & \cellcolor[HTML]{ACE1AF}x & \cellcolor[HTML]{ACE1AF}x & \cellcolor[HTML]{ACE1AF} & \cellcolor[HTML]{ACE1AF} & \cellcolor[HTML]{ACE1AF}\\ \hline
                \cellcolor[HTML]{ACE1AF} & \cellcolor[HTML]{ACE1AF} & \cellcolor[HTML]{ACE1AF} & \cellcolor[HTML]{ACE1AF} & \cellcolor[HTML]{ACE1AF} & \cellcolor[HTML]{ACE1AF}x & \cellcolor[HTML]{ACE1AF}x & \cellcolor[HTML]{ACE1AF}x \\ \hline
            \end{tabular}
            \caption{Dominating set of $\Johnson{8}{3}$.}
            \label{fig:domset-J(8,3)}
        \end{minipage}
    \end{figure}
    From these sets, we can build a dominating set of $\Johnson{9}{4}$ given by $D=D_1\cup \{w\cup\{9\}:w\in D_2\}$, represented in Figure~\ref{fig:domset-J(9,4)-no optimo} with which
    $$\dom(\Johnson{9}{4})\leq |D|=|D_1|+|D_2|=6+5=11.$$
    However, it is worth noting that $D$ is not a minimum dominating set of $\Johnson{9}{4}$. Indeed, $\dom(\Johnson{9}{4})\leq 9$, and a dominating set of cardinality $9$ is represented in Figure~\ref{fig:domset-J(9,4)}.
    
    \begin{figure}[ht]
        \centering
        \begin{tabular}{|c|c|c|c|c|c|c|c|c|}
            \hline
             1 & 2 & 3 & 4 & 5 & 6 & 7 & 8 & 9\\ \hline
             \cellcolor[HTML]{DAE8FC}x & \cellcolor[HTML]{DAE8FC} & \cellcolor[HTML]{DAE8FC}x & \cellcolor[HTML]{DAE8FC}x & \cellcolor[HTML]{DAE8FC}x & \cellcolor[HTML]{DAE8FC} & \cellcolor[HTML]{DAE8FC} & \cellcolor[HTML]{DAE8FC} & \\ \hline
             \cellcolor[HTML]{DAE8FC}x & \cellcolor[HTML]{DAE8FC} & \cellcolor[HTML]{DAE8FC}x & \cellcolor[HTML]{DAE8FC}x & \cellcolor[HTML]{DAE8FC} & \cellcolor[HTML]{DAE8FC}x & \cellcolor[HTML]{DAE8FC} & \cellcolor[HTML]{DAE8FC} & \\ \hline
             \cellcolor[HTML]{DAE8FC}x & \cellcolor[HTML]{DAE8FC} & \cellcolor[HTML]{DAE8FC} & \cellcolor[HTML]{DAE8FC} & \cellcolor[HTML]{DAE8FC}x & \cellcolor[HTML]{DAE8FC}x & \cellcolor[HTML]{DAE8FC}x & \cellcolor[HTML]{DAE8FC}& \\ \hline
             \cellcolor[HTML]{DAE8FC} & \cellcolor[HTML]{DAE8FC}x & \cellcolor[HTML]{DAE8FC} & \cellcolor[HTML]{DAE8FC} & \cellcolor[HTML]{DAE8FC}x & \cellcolor[HTML]{DAE8FC}x & \cellcolor[HTML]{DAE8FC} & \cellcolor[HTML]{DAE8FC}x & \\ \hline
             \cellcolor[HTML]{DAE8FC} & \cellcolor[HTML]{DAE8FC}x & \cellcolor[HTML]{DAE8FC}x & \cellcolor[HTML]{DAE8FC} & \cellcolor[HTML]{DAE8FC} & \cellcolor[HTML]{DAE8FC} & \cellcolor[HTML]{DAE8FC}x & \cellcolor[HTML]{DAE8FC}x & \\ \hline
             \cellcolor[HTML]{DAE8FC} & \cellcolor[HTML]{DAE8FC}x & \cellcolor[HTML]{DAE8FC} & \cellcolor[HTML]{DAE8FC}x & \cellcolor[HTML]{DAE8FC} & \cellcolor[HTML]{DAE8FC} & \cellcolor[HTML]{DAE8FC}x & \cellcolor[HTML]{DAE8FC}x & \\ \hline
            \cellcolor[HTML]{ACE1AF}x & \cellcolor[HTML]{ACE1AF}x & \cellcolor[HTML]{ACE1AF}x & \cellcolor[HTML]{ACE1AF} & \cellcolor[HTML]{ACE1AF} & \cellcolor[HTML]{ACE1AF} & \cellcolor[HTML]{ACE1AF} & \cellcolor[HTML]{ACE1AF} & \cellcolor[HTML]{FFFFC7}x \\ \hline
            \cellcolor[HTML]{ACE1AF}x & \cellcolor[HTML]{ACE1AF}x & \cellcolor[HTML]{ACE1AF} & \cellcolor[HTML]{ACE1AF}x & \cellcolor[HTML]{ACE1AF} & \cellcolor[HTML]{ACE1AF} & \cellcolor[HTML]{ACE1AF} & \cellcolor[HTML]{ACE1AF} & \cellcolor[HTML]{FFFFC7}x \\ \hline
            \cellcolor[HTML]{ACE1AF}x & \cellcolor[HTML]{ACE1AF}x & \cellcolor[HTML]{ACE1AF} & \cellcolor[HTML]{ACE1AF} & \cellcolor[HTML]{ACE1AF}x & \cellcolor[HTML]{ACE1AF} & \cellcolor[HTML]{ACE1AF} & \cellcolor[HTML]{ACE1AF} & \cellcolor[HTML]{FFFFC7}x \\ \hline
            \cellcolor[HTML]{ACE1AF} & \cellcolor[HTML]{ACE1AF} & \cellcolor[HTML]{ACE1AF}x & \cellcolor[HTML]{ACE1AF}x & \cellcolor[HTML]{ACE1AF}x & \cellcolor[HTML]{ACE1AF} & \cellcolor[HTML]{ACE1AF} & \cellcolor[HTML]{ACE1AF} & \cellcolor[HTML]{FFFFC7}x \\ \hline
            \cellcolor[HTML]{ACE1AF} & \cellcolor[HTML]{ACE1AF} & \cellcolor[HTML]{ACE1AF} & \cellcolor[HTML]{ACE1AF} & \cellcolor[HTML]{ACE1AF} & \cellcolor[HTML]{ACE1AF}x & \cellcolor[HTML]{ACE1AF}x & \cellcolor[HTML]{ACE1AF}x & \cellcolor[HTML]{FFFFC7}x \\ \hline
        \end{tabular}
        \caption{Dominating set of $\Johnson{9}{4}$.}
        \label{fig:domset-J(9,4)-no optimo}
    \end{figure}
\end{ex}

Repeating this idea, it is possible to recursively construct a dominating set of $\Johnson{n}{k}$ and thus obtain upper bounds for $\dom(\Johnson{n}{k})$.

\begin{theorem}\label{thm: recursive-upper-bound}
For all $m\geq 2k$ and $n\geq m+1$, we have
$$\dom(\Johnson{n}{k})\leq \dom(\Johnson{m}{k})+\sum_{i=m}^{n-1}\dom(\Johnson{i}{k-1}).$$
\end{theorem}

\begin{proof}
We will show by induction that for $n\geq m+1$, the graph $\Johnson{n}{k}$ has a dominating set of cardinality
$$f(n)=\dom(\Johnson{m}{k})+\sum_{i=m}^{n-1}\dom(\Johnson{i}{k-1}).$$

As mentioned above,
$$\dom(\Johnson{m+1}{k})\leq \dom(\Johnson{m}{k})+\dom(\Johnson{m}{k-1}),$$
so the claim holds for $n=m+1$.

Now let $n\geq m+2$ and suppose $J(n-1,k)$ has a dominating set $D_1$ with $|D_1|=f(n-1)$. Let $D_2$ be a minimum dominating set of $J(n-1,k-1)$. By Remark~\ref{rem:recursive-bound}, the set
$$D=D_1\cup \{w\cup\{n\}:w\in D_2\}$$
is a dominating set of $\Johnson{n}{k}$ and satisfies
\begin{align*}
\dom(\Johnson{n}{k})\leq |D|&=\overbrace{\left(\dom(\Johnson{m}{k})+\sum_{i=m}^{n-2}\dom(\Johnson{i}{k-1})\right)}^{=|D_1|=f(n-1)}+\overbrace{\dom(\Johnson{n-1}{k-1})}^{=|D_2|}\\
&=\dom(\Johnson{m}{k})+\sum_{i=m}^{n-1}\dom(\Johnson{i}{k-1})=f(n),
\end{align*}
which completes the proof.
\end{proof}

From Theorem~\ref{th:k=3} we know that $\dom(\Johnson{n}{3})=\phi_n$ for all $n\geq 6$. Moreover, the set $D$ represented in Figure~\ref{fig:domset-J(9,4)} is a dominating set of $\Johnson{9}{4}$, so $\dom(\Johnson{9}{4})\leq 9$ and, consequently, we can establish the following upper bound for $\dom(\Johnson{n}{4})$.

\begin{coro}\label{cor: upperbound-J(n,4)}
For $n\geq 9$ we have
$$\dom(\Johnson{n}{4})\leq \sum_{i=6}^{n-1}\phi_i-2.$$
\end{coro}

\begin{proof}
First, observe that the set $D$ represented in Figure~\ref{fig:domset-J(9,4)} is a dominating set of $\Johnson{9}{4}$, so
$$\dom(\Johnson{9}{4})\leq 9=\sum_{i=6}^{8}\phi_i-2.$$

Now, for $n\geq 10$, by Theorem~\ref{thm: recursive-upper-bound} with $m=9$, we have
$$\dom(\Johnson{n}{4})\leq \dom(\Johnson{9}{4})+\sum_{i=9}^{n-1}\underbrace{\dom(\Johnson{i}{3})}_{=\phi_i}\leq \sum_{i=6}^{8}\phi_i-2+\sum_{i=9}^{n-1}\phi_i=\sum_{i=6}^{n-1}\phi_i-2.$$
\end{proof}

\begin{figure}[ht]
    \centering
    \begin{tabular}{|c|c|c|c|c|c|c|c|c|}
        \hline
         1 & 2 & 3 & 4 & 5 & 6 & 7 & 8 & 9 \\ \hline
         x & x & x & x &  &  &  &  &  \\ \hline
         x & x & x &  & x &  &  &  &  \\ \hline
         x & x & x &  &  & x &  &  &  \\ \hline
          &  &  & x & x & x & x &  &  \\ \hline
          &  &  & x & x & x &  & x &  \\ \hline
          &  &  & x & x & x &  &  & x \\ \hline
         x &  &  &  &  &  & x & x & x \\ \hline
          & x &  &  &  &  & x & x & x \\ \hline
          &  & x &  &  &  & x & x & x \\ \hline
    \end{tabular}
    \caption{Dominating set $D$ of $\Johnson{9}{4}$.}
    \label{fig:domset-J(9,4)}
\end{figure}

Using the identity $\sum_{j=r}^{m}\binom{j}{r}=\binom{m+1}{r+1}$, we establish the following upper bound for $\dom(\Johnson{n}{k})$ with $k\geq 4$.

\begin{coro}\label{cor: upper-bound p(n)}
For $k\geq 4$ and $n\geq 2k+1$, we have
\[
\dom(\Johnson{n}{k})\leq \frac{1}{6}\binom{n}{k-1}-p_k(n),
\]
where $p_k$ is a polynomial in $n$ of degree $k-2$.
\end{coro}

\begin{proof}
We proceed by induction on $k$. First, consider $k=4$.

Recall that by Theorem~\ref{th:k=3}, for $j
\geq 6$ we have $\dom(\Johnson{j}{3})=\phi_j$, where
\begin{equation}\label{eq: phi_k}
\phi_j=
\begin{cases}
\frac{j^2-j+c_j}{12}, &\text{if $j$ is odd,}\\
\frac{j^2-2j+c_j}{12}, &\text{if $j$ is even,}
\end{cases}
\end{equation}
with $c_j$ as in~(\ref{eq: cn}). Note that $\phi_j$ is defined for $j\geq 6$. However, we will consider $\phi_j$ for $j\in [5]$ defined as in~(\ref{eq: phi_k}).


Let $a,b \in \N_0$ be such that $m=12a+b$ with $a,b\in\N_0$, $0\leq b<12$. Note that $\sum_{j=1}^{b}c_j < \frac{41}{6}b$ (see Table~\ref{tab:ck}), and since $c_j$ is defined modulo $12$, we obtain
\begin{equation}\notag
    \sum_{j=1}^{m}c_j=a\underbrace{\left(\sum_{j=1}^{12}c_j\right)}_{=82}+\underbrace{\sum_{j=1}^{b}c_j}_{< \frac{41}{6}b}< 82a+\frac{41}{6}b=\frac{41}{6}(12a+b)=\frac{41}{6}m.
\end{equation}

\begin{table}[ht]
    \centering
    \begin{tabular}{|c|c|c|c|}
    \hline
        $b$ & $c_b$ & $\displaystyle\sum_{j=1}^b c_j$ & $\frac{41}{6}b$\\[2pt] \hline
        $1$     & $0$   & $0$   &  $6+\frac{5}{6}$\\[2pt] \hline
        $2$     & $0$   & $0$   & $13+\frac{2}{3}$\\[2pt] \hline
        $3$     & $6$   & $6$   & $20+\frac{1}{2}$\\[2pt] \hline
        $4$     & $4$   & $10$  & $27+\frac{1}{3}$\\[2pt] \hline
        $5$     & $4$   & $14$  & $34+\frac{1}{6}$\\[2pt] \hline
        $6$     & $0$   & $14$  & $41$\\[2pt] \hline
        $7$     & $6$   & $20$  & $47+\frac{5}{6}$\\[2pt] \hline
        $8$     & $12$  & $32$  & $54+\frac{2}{3}$\\[2pt] \hline
        $9$     & $12$  & $44$  & $61+\frac{1}{2}$\\[2pt] \hline
        $10$    & $16$  & $60$  & $68+\frac{1}{3}$\\[2pt] \hline
        $11$    & $10$  & $70$  & $75+\frac{1}{6}$\\[2pt] \hline
        $12$    & $12$  & $82$  & $82$\\[2pt] \hline
    \end{tabular}
    \caption{Comparison between $\sum_{j=1}^{b}c_j$ and $\frac{41}{6}b$ for $1\leq b\leq 12$.}
    \label{tab:ck}
\end{table}

Therefore,
\begin{align*}
12\left(\sum_{j=1}^{n-1}\phi_j\right) &= \sum_{j=1}^{n-1} j(j-1)-\sum_{j=1}^{\left\lfloor \frac{n-1}{2} \right\rfloor}2j+\sum_{j=1}^{n-1} c_j\\
&= \sum_{j=2}^{n-1} 2 \binom{j}{2}-\sum_{j=1}^{\left\lfloor \frac{n-1}{2} \right\rfloor}2j+\sum_{j=1}^{n-1} c_j\\
&=2\binom{n}{3}- \underbrace{\left\lfloor \frac{n-1}{2} \right\rfloor\left(\left\lfloor \frac{n-1}{2} \right\rfloor+1\right)}_{\geq \frac{(n-1)^2-1}{4}}+\sum_{j=1}^{n-1} c_j\\
&\leq 2\binom{n}{3}-\frac{(n-1)^2-1}{4}+ \frac{41}{6}(n-1)
= 2\binom{n}{3}-\frac{3n^2-88n+82}{12}.
\end{align*}

From this we obtain
\[
\sum_{j=6}^{n-1}\phi_j\leq\frac{1}{6}\binom{n}{3}-\frac{3n^2-88n+82}{144}-\sum_{j=1}^{5}\phi_j
=\frac{1}{6}\binom{n}{3}-\frac{3n^2-88n+82}{144}-4.
\]

Consequently, by Corollary~\ref{cor: upperbound-J(n,4)}, for $n\geq 9$ it follows that
\[
\dom(\Johnson{n}{4})\leq \sum_{j=6}^{n-1}\phi_j-2
\leq \frac{1}{6}\binom{n}{3}-\frac{3n^2-88n+82}{144}-6
=\frac{1}{6}\binom{n}{3}-\frac{3n^2-88n+946}{144}.
\]

Therefore, the claim holds for $k=4$ with
\[
p_4(n)=\frac{3n^2-88n+946}{144}.
\]

Now, let $k\geq 5$ and suppose the claim holds for $k-1$. That is, for all $n\geq 2k-1$ we have
\[
\dom(\Johnson{n}{k-1})\leq \frac{1}{6}\binom{n}{k-2}-p_{k-1}(n),
\]
where $p_{k-1}$ is a polynomial in $n$ of degree $k-3$.

Applying Theorem~\ref{thm: recursive-upper-bound} with $m=2k$, we obtain
\begin{align*}
\dom(\Johnson{n}{k})
&\leq \dom(\Johnson{2k}{k})+\sum_{i=2k}^{n-1} \dom(\Johnson{i}{k-1})\\
&\leq \dom(\Johnson{2k}{k})+\sum_{i=2k}^{n-1}\left(\frac{1}{6}\binom{i}{k-2}-p_{k-1}(i)\right)\\
& = \sum_{i=k-2}^{n-1}\frac{1}{6}\binom{i}{k-2}-\left(\sum_{i=k-2}^{2k-1} \frac{1}{6}\binom{i}{k-2}+\sum_{i=2k}^{n-1}p_{k-1}(i)-\dom(\Johnson{2k}{k})\right)\\
&=\frac{1}{6}\binom{n}{k-1}-\left(\sum_{i=2k}^{n-1}p_{k-1}(i)+\frac{1}{6}\binom{2k}{k-1}-\dom(\Johnson{2k}{k})\right).
\end{align*}

We define $p_k$ as follows:
\begin{equation}\label{eq: p(n)}
p_k(n)=\sum_{i=2k}^{n-1}p_{k-1}(i)+\frac{1}{6}\binom{2k}{k-1}-\dom(\Johnson{2k}{k}).
\end{equation}

In this way we obtain
\[
\dom(\Johnson{n}{k})\leq \frac{1}{6}\binom{n}{k-1}-p_k(n),
\]
where $p_k$ is a polynomial in $n$ of degree $k-2$. This proves the claim.
\end{proof}

The polynomials $p_k(n)$ of Corollary~\ref{cor: upper-bound p(n)} can be computed recursively following (\ref{eq: p(n)}), with $p_4(n)=\frac{3n^2-88n+946}{144}$.

Note that the difference between the upper bound in~(\ref{eq: bounds-J(n,k)}) and the one established in Corollary~\ref{cor: upper-bound p(n)} is
\[
\frac{1}{k}\binom{n}{k-1}-\left(\frac{1}{6}\binom{n}{k-1}-p_k(n)\right)
=\frac{6-k}{6k}\binom{n}{k-1}+p_k(n).
\]

For $k=4$ and $k=5$, this is a polynomial in $n$ of degree $k-1$ that is positive for all $n$.
In the case $k=6$, this difference is exactly $p_5(n)$ and it is positive for all $n\geq 15$.
For $k\geq 7$, we obtain a polynomial in $n$ with negative leading coefficient.

Note that, although both the general bound~(\ref{eq: bounds-J(n,k)}) and the one in Corollary~\ref{cor: upper-bound p(n)} are, for fixed $k$, polynomials in $n$ of the same order (degree $k-1$), the fact that the leading coefficient of the difference between the two is negative for $k\geq 7$ implies that the improvement obtained does not hold for large $n$: from a certain value of $n$ onward (which grows with $k$), the difference becomes negative and the general bound~(\ref{eq: bounds-J(n,k)}) turns out to be tighter than the one given by Corollary~\ref{cor: upper-bound p(n)}. That is, our recursive construction effectively improves on the classical bound only for $k\leq 6$, and only in a bounded range of values of $n$ when $k\geq 7$.

In particular, we obtain the following improved upper bounds for $k\in\{4,5,6\}$.

\begin{coro}\label{cor: upper-bound r=4,5}
\begin{enumerate}
\item For $n\geq 9$, 
\[
\dom(\Johnson{n}{4})\leq \frac{1}{6}\binom{n}{3}-\frac{3n^2-88n+946}{144}.
\]

\item For $n\geq 11$, 
\[
\dom(\Johnson{n}{5})\leq \frac{1}{6}\binom{n}{4}-\frac{2n^3-91n^2+1981n-14726}{288}.
\]

\item For $n\geq 15$, 
\[
\dom(\Johnson{n}{6})\leq \frac{1}{6}\binom{n}{5}-\frac{3n^4-188n^3+6219n^2-94390n+499800}{1728}.
\]
\end{enumerate}
\end{coro}

\section{Total domination in Johnson graphs}\label{sec: tdom-Johnson}

In this section we study total dominating sets in Johnson graphs $J(n,k)$, with a particular focus on determining the total domination numbers for $k=2$ and $k=3$. Firstly, we state a monotonicity result by proving that $\tdom(\Johnson{n}{k})$ is non decreasing with respect to $n$. We determine the value of $\tdom(\Johnson{n}{2})$. Next, we address the case $k=3$, proposing a translation of total dominating sets into edge triangle coverings (with an additional property) of graphs with order $n$ and independence number at most $2$. With this approach, we establish the value of $\tdom(\Johnson{n}{3})$ for every $n\geq 10$.

Let start with the mentioned monotonicity result.

\begin{theorem}\label{thm: monotonicity total Johnson}
    If $n\geq 2k$, then $\tdom(\Johnson{n}{k})\leq\tdom(\Johnson{n+1}{k})$.
\end{theorem}

\begin{proof}
    Let $D$ be a minimum total dominating set of $\Johnson{n+1}{k}$, $D=D_0\cup \hat{D}$, where $D_0=\{w\in D:n+1\notin w\}$ and $\hat{D}=\{w\in D:n+1\in D\}$. 
    We index the vertices in $\hat{D}$, $\hat{D}=\{w_1,\ldots,w_r\}$, and iteratively build sets $D_j$ for $j\in [r]$ as follows. For each $j\in[r]$, let $\hat{w_j}=w_j\setminus \{n+1\}$. If for every $x\in [n]\setminus \hat{w_j}$, $\hat{w_j}\cup\{x\}\in D_{j-1}$, then $D_j:=D_{j-1}$. Otherwise, let $x$ be the largest number in $[n]\setminus \hat{w_j}$ such that $\hat{w_j}\cup\{x\}\notin D_{j-1}$. We define $w_j'=\hat{w_j}\cup\{x\}$, and $D_{j}:=D_{j-1}\cup\{w_j'\}$.
    
    Let $D'=D_r$. We will show that $D'$ is a total dominating set of $\Johnson{n}{k}$. Let $v\in V(\Johnson{n}{k})$. As $v\in\Johnson{n+1}{k}$, there exists $w\in D\cap N_{n+1}(v)$. 
    If $w\in D_0$, then $w\in D_0\cap N_n(v)\subseteq D'\cap N_n(v)$. Otherwise, $w=\hat{w}\cup \{n+1\}$, and $v\cap w=\hat{w}$. By the construction of $D'$, $\hat{w}\cup\{x\}\in D'$ for some $x\in [n]$.
    If for some $x\in[n]\setminus v$, $u=\hat{w}\cup\{x\}\in D'$, then $u\in D'\cap N_n(v)$. 
    Thus, assume $v$ is the only such vertex in $D'$, i.e.\ the only vertex in $D'$ that contains all elements of $\hat{w}$ is $v$. Hence $v=\hat{w}\cup\{x\}$ and by construction $x=\max([n]\setminus w)=\max([n]\setminus \hat{w})$. Besides, it follows that no vertex in $D_0$ contain $\hat{w}$.

    On the other hand, since $w$ is totally dominated by $D$, there exists $u\in D\cap N_{n+1}(w)$. If $u\in D_0$, then as $u$ and $w$ are adjacent and $n+1\notin u$, it follows $\hat{w}\subseteq u$.  which is not possible. Thus we may assume $u\notin D_0$. 
     So $u=\hat{u}\cup \{n+1\}$ with $|\hat{u}\cap \hat{w}|=k-2$. By construction of $D'$, $u'=\hat{u}\cup\{y\}\in D'$, where $y=\max([n]\setminus u)=\max([n]\setminus \hat{u})$.
   If $x=\max([n]\setminus \hat{w})=\max([n]\setminus \hat{u})=y$, then $|u'\cap v|=|(\hat{w}\cap \hat{u})\cup\{x\}|=k-1$. If $x<y$, then $y\in \hat{w}$, so $|u'\cap v|=|\hat{w}|=k-1$. Finally, if $x>y$, then $x\in \hat{u}$, so $|u'\cap v|=|\hat{u}|=k-1$. In any case, $u'\in D'\cap N_n(v)$.
    
    Therefore, $D'$ is a total dominating set of $\Johnson{n}{k}$ with $|D'|\leq|D|$ and so $\tdom(\Johnson{n}{k})\leq |D'|\leq |D|=\tdom(\Johnson{n+1}{k})$, which completes the proof.
\end{proof}


\begin{theorem}\label{teo: tdom Johnson r=2}
    For $n\geq 4$, $\dom_t(\Johnson{n}{2})=\left\lceil\frac{2}{3}(n-1)\right\rceil$.
\end{theorem}
\begin{proof}
    Let $D$ be a minimum total dominating set of $\Johnson{n}{2}$.
    We define a bipartite graph $G$ with $V(G)=D\cup [n]$, where for $u\in D$ and $x\in [n]$, $u$ and $x$ are adjacent in $G$ if and only if $x\in u$. Let $X\subseteq [n]$ be the set of all $x\in [n]$ that occur in more than one vertex of $D$, and $H=G[D\cup X]$, that is, the subgraph of $G$ induced by $D\cup X$.

    Observe that if $u\in D$, since $D$ is a total dominating set, there exists $v\in D\setminus\{u\}$ such that $|u\cap v|=1$. Then, for some $x\in X$, we have $u\cap v=\{x\}$, so $\text{deg}_H(u)\geq 1$.
    Also, for each $x\in X$, we have $\text{deg}_H(x)\geq 2$.
    Therefore,
    $$2\left(\sum_{x\in X}\text{deg}_H(x)\right)=2|E(H)|=\sum_{u\in D}\text{deg}_H(u)+\sum_{x\in X}\text{deg}_H(x)\geq |D|+2|X|.$$
    Then
    $$\sum_{x\in X} (\text{deg}_H(x)-1)\geq \frac{|D|}{2}.$$
    On the other hand, let $X'=\{x\in [n]:x\in u\text{ for some }u\in D\}=\{x\in u:\text{deg}_G(x)\geq 1\}$. Note that if $x,y\in [n]\setminus X'$, then the vertex $\{x,y\}\in V(\Johnson{n}{2})$ has no neighbor in $D$, contradicting that $D$ is dominating. Therefore, there cannot be two elements outside $X'$, and so $|X'|\geq n-1$. Then
    $$2|D|=\sum_{u\in D}\text{deg}_G(u)=\sum_{x\in [n]}\text{deg}_G(x)= |X'|+\sum_{x\in X}(\text{deg}_H(x)-1)\geq n-1+\frac{|D|}{2},$$
    so
    $$|D|\geq \left\lceil\frac{2(n-1)}{3}\right\rceil.$$
    
    Now let us see that there exists a total dominating set of $\Johnson{n}{2}$ of cardinality $\left\lceil\frac{2(n-1)}{3}\right\rceil$. Let $k,\ell$ be non-negative integers such that $n=3k+\ell$ with $0\leq \ell\leq 2$. We define, for $\ell\in \{0,1\}$, the set
    $$D=\bigcup_{i=0}^{k-1}\left\{\{3i+1,3i+2\},\{3i+2,3i+3\}\right\},$$
    and for $\ell=2$, the set
    $$D=\left(\bigcup_{i=0}^{k-1}\left\{\{3i+1,3i+2\},\{3i+2,3i+3\}\right\}\right)\cup\{n-2,n-1\}.$$
    It is easy to see that $D$ is a total dominating set of $\Johnson{n}{2}$ of cardinality $\left\lceil\frac{2}{3}(n-1)\right\rceil$.
\end{proof}

\subsection{Johnson graphs \texorpdfstring{$\Johnson{n}{3}$}{J(n,3)}}

We focus now on Johnson graphs $\Johnson{n}{3}$. We will state a translation of total dominating sets of these graphs into edge triangle coverings of graphs of order $n$ and independence number at most $2$. In order to do so, we introduce the following definitions.

Given a graph $G$ of order $n$, a collection $C$ of triangles of $G$ is an \textit{edge triangle covering} if for each $uv\in E(G)$, there exists $T\in C$ such that $\{u,v\}\subseteq T$. Given two triangles $T,S$ of $G$, we say that $T$ and $S$ \textit{overlap} if $|T\cap S|=2$. This is, if they have exactly two vertices in common.
We say that $\mathcal C$ is a \textit{covering by edge-overlapping triangles} (CEOT), if $\mathcal C$ is an edge triangle covering such that for each $T\in C$ there exists $S\in C$ such that $T$ and $S$ overlap. We denote by $\rho_{\Delta}^s(G)$ the minimum cardinality of a CEOT of $G$. If $G$ does not admit a CEOT, $\rho_{\Delta}^s(G)=\infty$.

In~\cite{CDT-2025}, it was shown that a set $D \subseteq V(J(n,3))$ is a dominating set of $J(n,3)$ if and only if $D$ is an edge triangle covering of a graph $G$ of order $n$ with $\alpha(G) \leq 2$. This correspondence was then used to determine $\gamma(J(n,3))$. Similarly, we will establish a correspondence between total dominating sets of $J(n,3)$ and CEOTs of graphs of order $n$ with $\alpha(G) \leq 2$, which will then allow us to determine the exact value of $\tdom(J(n,3))$. For this we recall the following definition introduced in~\cite{CDT-2025}. Let $D$ be an arbitrary family of subsetes of $[n]$ of cardinality 3. Denote by $G^D$ the graph with vertex set $[n]$ and two vertices $x$ and $y$ being adjacent in $G^D$ if there exists some $u \in D$ that contains both $x$ and $y$.

\begin{prop}\label{prop:covering-overlapping}
    Let $D\subseteq V(\Johnson{n}{3})$. Then $D$ is a total dominating set of $\Johnson{n}{3}$ if and only if $D$ is a CEOT of some graph $G$ of order $n$ with $\alpha(G)\leq 2$.
\end{prop}
\begin{proof}
        Let $D$ be a total dominating set of $\Johnson{n}{3}$. It follows that each $u\in D$ is a triangle of $G^D$, and each edge of $G^D$ belongs to at least one triangle in $D$. Thus $D$ is an edge triangle covering of $G^D$. Moreover, as $D$ is a total dominating set of $\Johnson{n}{3}$, for each $u\in D$, there exists $v\in D$ with $|u\cap v|=2$. This is, there exists a triangle $v\in D$ such that $u$ and $v$ overlap. Therefore $D$ is a CEOT of $G^D$.
        Suppose that $\alpha(G^D) \geq 3$. Then, there exist three vertices $x$, $y$ and $z$ pairwise non adjacent in $G^D$. By definition of $G^D$, $\{x,y,z\} \in V(\Johnson{n}{3})$ is not in $D$ and has no neighbor in $D$, arising to a contradiction. So $\alpha(G^D) \leq 2$. Therefore, $D$ is a CEOT of $G^D$ which has order $n$ and $\alpha(G^D)\leq 2$.

        Conversely, let $D'$ be a CEOT of a graph $G$ of order $n$ with $\alpha(G)\leq 2$, and let $u\in V(\Johnson{n}{3})$. If $u\in D'$, since the triangles in $D'$ overlap, there exists $v\in D'$ with $|u\cap v|=2$. If $u\notin D'$, then as $\alpha(G)\leq 2$, there exist $a,b\in u$ such that $ab\in E(G)$. Since $D'$ is an edge triangle covering, there exists $v\in D'$ such that $\{a,b\}\subseteq v$, so $|u\cap v|=2$. That is, in either case, there exists $v\in D'$ such that $u$ and $v$ are adjacent. Therefore, $D'$ is a total dominating set of $\Johnson{n}{3}$.
    \end{proof}

As a consequence of this result, we have that determining $\gamma_t(J(n,3))$ is then equivalent to determining 
\begin{equation}\label{eq: t_n^s}
    t_n^s=\min\{\rho_{\Delta}^s(G):G\text{ a graph with }|V(G)|=n, \alpha(G)\leq 2\}.
\end{equation}
That is, $\tdom(\Johnson{n}{3})=t_n^s$. Moreover, if $D$ is a $\tdom$-set of $J(n,3)$, then
\begin{equation}\label{eq: tn D dom minimo}
    t_n^s=|D|=\rho_{\Delta}^s(G^D).
\end{equation}

\begin{figure}[ht]
    \centering
    \begin{tikzpicture}[scale=0.75]
    \begin{scope}
        
        \def\r{3}
        \def\rr{6}
    
        \node (245) at ({15+0*30}:\r) {};
        \node (246) at ({15+1*30}:\r) {};
        \node (256) at ({15+2*30}:\r) {};
        \node (134) at ({15+3*30}:\r) {};
        \node (135) at ({15+4*30}:\r) {};
        \node (136) at ({15+5*30}:\r) {};
        \node (145) at ({15+6*30}:\r) {};
        \node (146) at ({15+7*30}:\r) {};
        \node (156) at ({15+8*30}:\r) {};
        \node (234) at ({15+9*30}:\r) {};
        \node (235) at ({15+10*30}:\r) {};
        \node (236) at ({15+11*30}:\r) {};
        
        \node (123) at ({135+0*30}:\rr) {};
        \node (124) at ({135+1*30}:\rr) {};
        \node (125) at ({135+2*30}:\rr) {};
        \node (126) at ({135+3*30}:\rr) {};
        
        \node (456) at ({45-0*30}:\rr) {};
        \node (356) at ({45-1*30}:\rr) {};
        \node (346) at ({45-2*30}:\rr) {};
        \node (345) at ({45-3*30}:\rr) {};

        \draw (123) -- (124);
        \draw (123) -- (125);
        \draw (123) -- (126);
        \draw (124) -- (125);
        \draw (124) -- (126);
        \draw (125) -- (126);
        
        \draw (123) -- (134);
        \draw (123) -- (135);
        \draw (123) -- (136);
        \draw (124) -- (134);
        \draw (124) -- (145);
        \draw (124) -- (146);
        \draw (125) -- (135);
        \draw (125) -- (145);
        \draw (125) -- (156);
        \draw (126) -- (136);
        \draw (126) -- (146);
        \draw (126) -- (156);
        
        \draw (123) -- (234);
        \draw (123) to[bend right = 10] (235);
        \draw (123) -- (236);
        \draw (124) -- (234);
        \draw (124) -- (245);
        \draw (124) -- (246);
        \draw (125) -- (235);
        \draw (125) to[bend right = 10] (245);
        \draw (125) -- (256);
        \draw (126) -- (236);
        \draw (126) to[bend right = 10] (246);
        \draw (126) -- (256);

        \draw (134) -- (135);
        \draw (134) -- (136);
        \draw (134) -- (145);
        \draw (134) -- (146);
        \draw (135) -- (136);
        \draw (135) -- (145);
        \draw (135) -- (156);
        \draw (136) -- (146);
        \draw (136) -- (156);
        \draw (145) -- (146);
        \draw (145) -- (156);
        \draw (146) -- (156);

        \draw (345) -- (346);
        \draw (345) -- (356);
        \draw (345) -- (456);
        \draw (346) -- (356);
        \draw (346) -- (456);
        \draw (356) -- (456);

        \draw (345) -- (134);
        \draw (345) to[bend left = 10] (135); 
        \draw (345) -- (145);
        \draw (346) -- (134);
        \draw (346) to[bend left = 10] (136); 
        \draw (346) -- (146);
        \draw (356) -- (135);
        \draw (356) -- (136);
        \draw (356) -- (156);
        \draw (456) -- (145);
        \draw (456) to[bend left = 10] (146); 
        \draw (456) -- (156);

        \draw (345) -- (234);
        \draw (345) -- (235);
        \draw (345) -- (245);
        \draw (346) -- (234);
        \draw (346) -- (236);
        \draw (346) -- (246);
        \draw (356) -- (235);
        \draw (356) -- (236);
        \draw (356) -- (256);
        \draw (456) -- (245);
        \draw (456) -- (246);
        \draw (456) -- (256);

        \draw (234) -- (235);
        \draw (234) -- (236);
        \draw (234) -- (245);
        \draw (234) -- (246);
        \draw (235) -- (236);
        \draw (235) -- (245);
        \draw (235) -- (256);
        \draw (236) -- (246);
        \draw (236) -- (256);
        \draw (245) -- (246);
        \draw (245) -- (256);
        \draw (246) -- (256);

        \draw (134) -- (234);
        \draw (135) -- (235);
        \draw (136) -- (236);
        \draw (145) -- (245);
        \draw (146) -- (246);
        \draw (156) -- (256);

        \foreach \x in {1,2,3,4,5,6}{
        \foreach \y in {2,...,6}{ 
        \foreach \z in {3,...,6}{ 
        \ifthenelse{\x < \y}{
        \ifthenelse{\y < \z}{
        \begin{scope}[shift={(\x\y\z)},rotate=90]
            \filldraw[fill=white, thick] (0,0) circle (0.25);
            \foreach \w in {0,1,...,5}{\draw (0,0) -- ({\w*60}:0.25);}
            \filldraw[fill=gray!70] (0,0) -- ({-(\x-1)*360/6}:0.25) arc ({-(\x-1)*360/6}:{-\x*360/6}:0.25) -- cycle;
            \filldraw[fill=gray!70] (0,0) -- ({-(\y-1)*360/6}:0.25) arc ({-(\y-1)*360/6}:{-\y*360/6}:0.25) -- cycle;
            \filldraw[fill=gray!70] (0,0) -- ({-(\z-1)*360/6}:0.25) arc ({-(\z-1)*360/6}:{-\z*360/6}:0.25) -- cycle;
        \end{scope}
        }{}}{}
        }
        }
        }
        \begin{scope}[shift={(123)},rotate=90]
            \filldraw[fill=white, thick] (0,0) circle (0.25);
            \foreach \w in {0,1,...,5}{\draw (0,0) -- ({\w*60}:0.25);}
            \foreach \x in {1,2,3}{
            \filldraw[fill=red] (0,0) -- ({-(\x-1)*360/6}:0.25) arc ({-(\x-1)*360/6}:{-\x*360/6}:0.25) -- cycle;}
        \end{scope}
        \begin{scope}[shift={(456)},rotate=90]
            \filldraw[fill=white, thick] (0,0) circle (0.25);
            \foreach \w in {0,1,...,5}{\draw (0,0) -- ({\w*60}:0.25);}
            \foreach \x in {4,5,6}{
            \filldraw[fill=blue] (0,0) -- ({-(\x-1)*360/6}:0.25) arc ({-(\x-1)*360/6}:{-\x*360/6}:0.25) -- cycle;}
        \end{scope}
        \begin{scope}[shift={(234)},rotate=90]
            \filldraw[fill=white, thick] (0,0) circle (0.25);
            \foreach \w in {0,1,...,5}{\draw (0,0) -- ({\w*60}:0.25);}
            \foreach \x in {2,3,4}{
            \filldraw[fill=green!50!black] (0,0) -- ({-(\x-1)*360/6}:0.25) arc ({-(\x-1)*360/6}:{-\x*360/6}:0.25) -- cycle;}
        \end{scope}
        \begin{scope}[shift={(156)},rotate=90]
            \filldraw[fill=white, thick] (0,0) circle (0.25);
            \foreach \w in {0,1,...,5}{\draw (0,0) -- ({\w*60}:0.25);}
            \foreach \x in {1,5,6}{
            \filldraw[fill=orange] (0,0) -- ({-(\x-1)*360/6}:0.25) arc ({-(\x-1)*360/6}:{-\x*360/6}:0.25) -- cycle;}
        \end{scope}
    \end{scope}
    \begin{scope}[shift={(11,0)}]
        \draw[red, line width=2, opacity=0.7] ({120-1*60}:2.5) -- ({120-2*60}:2.5) -- ({120-3*60}:2.5) -- cycle;
        \draw[blue, line width=2, opacity=0.7] ({120-4*60}:2.5) -- ({120-5*60}:2.5) -- ({120-6*60}:2.5) -- cycle;
        \draw[green!50!black, line width=2, opacity=0.7] ({120-4*60}:2.5) -- ({120-2*60}:2.5) -- ({120-3*60}:2.5) -- cycle;
        \draw[orange, line width=2, opacity=0.7] ({120-1*60}:2.5) -- ({120-5*60}:2.5) -- ({120-6*60}:2.5) -- cycle;
        \foreach \x in {1,...,6}{
        \filldraw (\x*60:2.5) circle (3pt);
        \node at ({120-\x*60}:2.9) {\small $\x$};}
    \end{scope}
    \end{tikzpicture}
    \caption{Minimum total dominating set of $\Johnson{6}{3}$.}
    \label{fig: conj tdom min J(6,3)}
\end{figure}
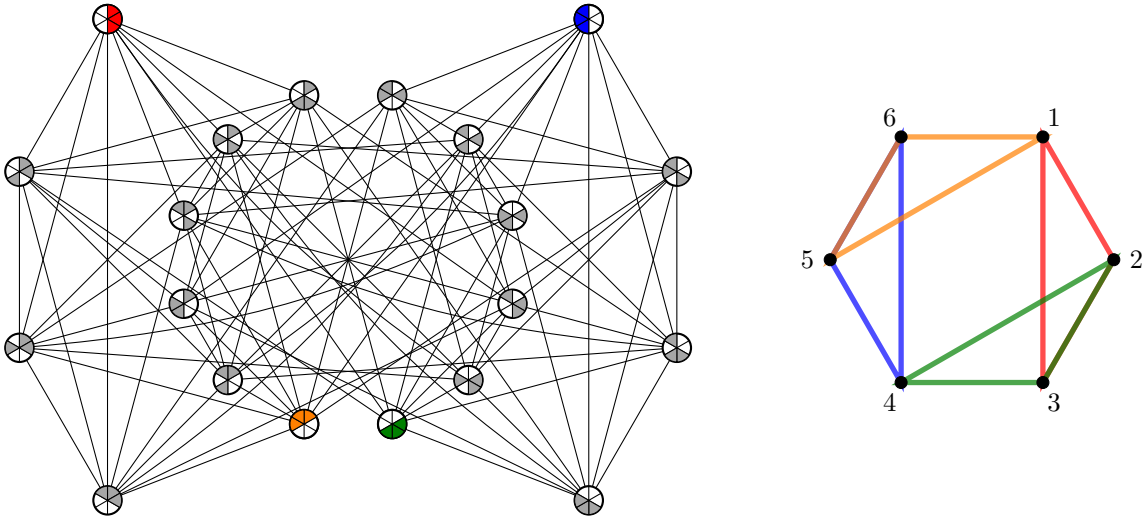

It can be easily checked (either by a computer or case analysis) that $\tdom(J(6,3))=4$ and $\tdom(J(7,3))=5$. The minimum total dominating set $D=\{\{1,2,3\},\{2,3,4\},\{4,5,6\},$ $\{1,5,6\}\}$ of $J(6,3)$ is illustrated in Figure~\ref{fig: conj tdom min J(6,3)}. Moreover, the minimum total dominating set of $J(7,3)$ is depicted in Figure~\ref{fig: conj tdom min J(7,3)}. In both cases we presented total dominating sets by coverings by edge-overlapping triangles of a graph $G$ of order $n \in \{6,7\}$ with $\alpha(G)=2$.

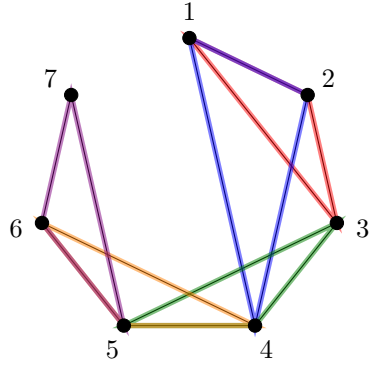
\begin{figure}[ht]
    \centering
    \begin{tikzpicture}
        \begin{scope}
                \draw ({90-(1-1)*360/7}:2) -- ({90-(2-1)*360/7}:2) -- ({90-(3-1)*360/7}:2) -- cycle;
                \draw ({90-(1-1)*360/7}:2) -- ({90-(4-1)*360/7}:2) -- ({90-(2-1)*360/7}:2);
                \draw ({90-(4-1)*360/7}:2) -- ({90-(5-1)*360/7}:2) -- ({90-(3-1)*360/7}:2) -- cycle;
                \draw ({90-(5-1)*360/7}:2) -- ({90-(6-1)*360/7}:2) -- ({90-(7-1)*360/7}:2) -- cycle;
                \draw ({90-(6-1)*360/7}:2) -- ({90-(4-1)*360/7}:2);

                \draw[red, line width=2, opacity=0.5] ({90-(1-1)*360/7}:2) -- ({90-(2-1)*360/7}:2) -- ({90-(3-1)*360/7}:2) -- cycle;
                \draw[blue, line width=2, opacity=0.5] ({90-(1-1)*360/7}:2) -- ({90-(2-1)*360/7}:2) -- ({90-(4-1)*360/7}:2) -- cycle;
                \draw[green!50!black, line width=2, opacity=0.5] ({90-(3-1)*360/7}:2) -- ({90-(4-1)*360/7}:2) -- ({90-(5-1)*360/7}:2) -- cycle;
                \draw[orange, line width=2, opacity=0.5] ({90-(4-1)*360/7}:2) -- ({90-(5-1)*360/7}:2) -- ({90-(6-1)*360/7}:2) -- cycle;
                \draw[violet, line width=2, opacity=0.5] ({90-(5-1)*360/7}:2) -- ({90-(6-1)*360/7}:2) -- ({90-(7-1)*360/7}:2) -- cycle;
                
                \foreach \x in {1,...,7}{
                \filldraw ({90-(\x-1)*360/7}:2) circle (2.5pt);
                \node at ({90-(\x-1)*360/7}:2.35) {\small$\x$};
                }
            \end{scope}
    \end{tikzpicture}
    \caption{Graph $G^D$ obtained from total dominating set $D$ of $J(7,3)$ of cardinality 5.}
    \label{fig: conj tdom min J(7,3)}
    \end{figure}

\subsubsection{An upper bound for \texorpdfstring{$\tdom(\Johnson{n}{3})$}{tdom(J(n,3))}}

Let $G$ be a graph and $D=\{D_1,D_2,\ldots,D_\ell\}$ a collection of subgraphs of $G$, not necessarily induced, each one isomorphic to the \emph{diamond} graph $K_4-e$ (see Figure \ref{fig:diamante}).
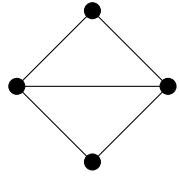
\begin{figure}[ht]
    \centering
    \begin{tikzpicture}
        \draw (-1,0) -- (0,1) -- (1,0) -- (0,-1) -- cycle;
        \draw (-1,0) -- (1,0);
        \filldraw (0,1) circle (3pt);
        \filldraw (0,-1) circle (3pt);
        \filldraw (1,0) circle (3pt);
        \filldraw (-1,0) circle (3pt);
    \end{tikzpicture}
    \caption{Diamond graph ($K_4-e$).}
    \label{fig:diamante}
\end{figure}
We say that $D$ is an \textit{edge diamond covering} of $G$ if $E(G)=\bigcup_{i=1}^{\ell} E(D_i)$. We say that $D$ is an \textit{edge diamond packing} of $G$ if the edges of the graphs $D_i$ are pairwise disjoint.
The \textit{diamond covering number} of $G$, denoted $\text{cov}(G,\text{diam})$, is the minimum cardinality of an edge diamond covering of $G$, and the \textit{diamond packing number} of $G$, denoted $\text{pack}(G,\text{diam})$, is the maximum cardinality of an edge diamond packing of $G$.
If $D$ is an edge diamond covering that is also a packing, we say that $D$ is a \textit{diamond decomposition} of $G$. The study of edge diamond coverings and packings, as well as those by other graphs of order at most $4$, has been studied in complete graphs (see, for example, \cite{bermond1977g,hoffman1996maximum,RODITTY1983231}).

Note that if $D^*$ is an edge diamond covering of a graph $G$, then it is possible to obtain a CEOT $D$ of $G$ with cardinality $|D|=2|D^*|$ by considering, for each diamond in $D^*$, two triangles as in Figure~\ref{fig:triangulos a partir de diamante}.
\begin{figure}[ht]
    \centering
    \begin{tikzpicture}
        \draw (-1,0) -- (0,1) -- (1,0) -- (0,-1) -- cycle;
        \draw (-1,0) -- (1,0);
        \filldraw (0,1) circle (3pt) node[yshift=8] {$x_2$};
        \filldraw (0,-1) circle (3pt) node[yshift=-8] {$x_4$};
        \filldraw (1,0) circle (3pt) node[xshift=8] {$x_3$};
        \filldraw (-1,0) circle (3pt) node[xshift=-8] {$x_1$};
        \node at (0,-2) {$d$};

        \begin{scope}[shift={(4,0)}]
        \draw (-1,0) -- (0,1) -- (1,0) -- cycle;
        \filldraw (0,1) circle (3pt) node[yshift=8] {$x_2$};
        \filldraw (1,0) circle (3pt) node[xshift=8] {$x_3$};
        \filldraw (-1,0) circle (3pt) node[xshift=-8] {$x_1$};
        \node at (0,-2) {$T_1$};
        \end{scope}
        
        \begin{scope}[shift={(8,0)}]
        \draw (-1,0) -- (0,-1) -- (1,0) -- cycle;
        \filldraw (0,-1) circle (3pt) node[yshift=-8] {$x_4$};
        \filldraw (1,0) circle (3pt) node[xshift=8] {$x_3$};
        \filldraw (-1,0) circle (3pt) node[xshift=-8] {$x_1$};
        \node at (0,-2) {$T_2$};
        \end{scope}
    \end{tikzpicture}
    \caption{$T_1$ and $T_2$ triangles obtained from the diamond $d$.}
    \label{fig:triangulos a partir de diamante}
\end{figure}
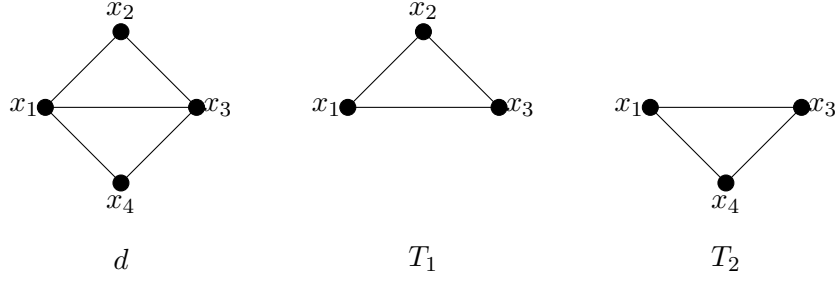
If $D^*$ is a minimum edge diamond covering of $G$, we obtain
\begin{equation}\label{eq: cota CEOT con covering}
    \rho_{\Delta}^s(G)\leq |D|=2\,|D^*|= 2\,\text{cov}(G,\text{diam}).
\end{equation}
Now, let $m$ be a natural number with $4\leq m\leq n-4$, and let $G$ be the disjoint union of $K_m$ and $K_{n-m}$. Clearly $G$ is a graph of order $n$ and $\alpha(G)\leq 2$. Then, from an edge diamond covering of $G$, which will be an edge diamond covering of each of its connected components, we obtain a CEOT of this graph $G$ of order $n$ with $\alpha(G)\leq 2$, which by Proposition \ref{prop:covering-overlapping} is a total dominating set of $\Johnson{n}{3}$. Therefore
\begin{equation}\notag
    \tdom(\Johnson{n}{3})\leq \rho_{\Delta}^s(G)=\rho_{\Delta}^s(K_m)+\rho_{\Delta}^s(K_{n-m}).
\end{equation}
Given $n,m\in\N$ with $n\geq 8$ and $m\in[4..(n-4)]$, we define $f_n(m)$ as
\begin{equation}\label{eq: f_n(m)}
    f_n(m)=\rho_{\Delta}^s(K_m)+\rho_{\Delta}^s(K_{n-m}),
\end{equation}
and $\varphi_n$ as 
\begin{equation}\label{eq: varphi_n}
    \varphi_n=\min\left\{f_n(m):m\in [4..(n-4)]\right\}.
\end{equation}
Since the definition of $f_n(m)$ directly implies that $f_n(m)=f_n(n-m)$, it follows that for any $n \geq 8$, $\varphi_n=\min\left\{f_n(m):m\in [4..\left\lfloor\frac{n}{2}\right\rfloor]\right\}$.
Therefore, we have the following upper bound for the total domination number of $\Johnson{n}{3}$.

\begin{prop}\label{prop: cota sup tdom varphi_n}
    For $n\geq 8$, we have
    $$\tdom(\Johnson{n}{3})\leq \varphi_n.$$
\end{prop}

We begin by determining $\rho_{\Delta}^s(K_n)$, which will allow us to determine $\varphi_n$. The following result is due to \cite{RODITTY1983231}.

\begin{theorem}[\cite{RODITTY1983231}]\label{thm: cov-diam}
For $n\geq 6$, the following hold:
    \begin{enumerate}
        \item $\text{cov}(K_n,\text{diam})=\left\lceil\dfrac{n(n-1)}{10}\right\rceil.$
        \item $\text{pack}(K_n,\text{diam})=\left\lfloor\dfrac{n(n-1)}{10}\right\rfloor.$
    \end{enumerate}
\end{theorem}

\begin{remark}\label{obs: leave packing diamond}
    Let $n\geq 6$ and let $S$ be a maximum edge diamond packing of $K_n$, and $P(S)$ the set of edges of $K_n$ not covered by any triangle of $S$.
    In \cite{hoffman1996maximum} it is established that 
    $$|P(S)|=\begin{cases}
        0, &\text{if }n\equiv 0,1 \mod 5,\\
        1, &\text{if }n\equiv 2,4 \mod 5 \text{ and }n\notin \{7,9\},\\
        3, &\text{if }n\equiv 3 \mod 5,\\
        6, &\text{if }n\in \{7,9\}.
    \end{cases}$$
    Moreover, if $n\equiv 3 \mod 5$, there exists a maximum packing $S$ for which $P(S)=\{ab,bc,cd\}$, for $a,b,c,d\in [n]$. 
\end{remark}

Let $n\in\N$ with $n\geq 4$, and let $D^*$ be an edge diamond covering of $K_n$. Then
$|D^*|\geq \text{cov}(K_n,\text{diam}) \geq \frac{n(n-1)}{10}=\frac{|E(K_n)|}{5}.$
Now, let $D$ be the CEOT obtained from $D^*$, so that $|D|=2|D^*|$. It follows that
$$|E(K_n)|\leq 5|D^*|=\frac{5}{2}\,|D|.$$
In the following lemma, we show that this inequality holds for every CEOT $D$ of a graph $G$.

\begin{lemma}\label{lem: edges CEOT}
    Given a graph $G$, if $D$ is a CEOT of $G$, then $|E(G)|\leq \frac{5}{2}|D|$. 
\end{lemma}
\begin{proof}
    Let $D$ be a CEOT of $G$, and $E'$ the set of edges of $G$ that are in more than one triangle of $D$.
    We define a bipartite graph $H$ with $V(H)=D\cup E'$, where for $t\in D$ and $e\in E'$, $t$ and $e$ are adjacent in $H$ if and only if $e$ is an edge in $t$. 
    Observe that by the definition of $E'$, $\text{deg}_H(e)\geq 2$ for each $e\in E'$. Since $D$ is a CEOT, $\text{deg}_H(t)\geq 1$ for each $t\in D$. Moreover, since $H$ is bipartite with bipartition $D\cup E'$, we have $\sum_{t\in D}\text{deg}_H(t)=\sum_{e\in E'}\text{deg}_H(e)$. Then, we have
    \begin{equation}\notag
    \begin{aligned}
        2\left(\sum_{e\in E'}\text{deg}_H(e)\right)=\sum_{t\in D}\overbrace{\text{deg}_H(t)}^{\geq 1}+\sum_{e\in E'}\overbrace{\text{deg}_H(e)}^{\geq 2}\geq |D|+2|E'|. 
    \end{aligned}
    \end{equation}

    On the other hand, if $\hat{E}$ is the set of edges appearing in exactly one triangle, we have
    that
    \begin{equation}\notag
    \begin{aligned}
        3|D|&=\sum_{d\in D} |E(d)|= |\hat{E}|+\sum_{e\in E'}\text{deg}_H(e)= \overbrace{|\hat{E}|+|E'|}^{=|E(G)|}+\sum_{e\in E'}(\text{deg}_H(e)-1)\\
        &= |E(G)|+\sum_{e\in E'}(\text{deg}_H(e)-1)=|E(G)|+\underbrace{\left(\sum_{e\in E'}\text{deg}_H(e)-|E'|\right)}_{\geq \frac{|D|}{2}}\geq |E(G)|+\frac{|D|}{2}
    \end{aligned}
    \end{equation}
    Therefore, 
    $$|E(G)|\leq 3|D|-\frac{|D|}{2}=\frac{5}{2}|D|.$$
\end{proof}

The above results can be used to derive upper bounds for $J\tdom((n,3))$. For example, Proposition~\ref{prop: cota sup tdom varphi_n} implies that $\tdom(\Johnson{10}{3})\leq \allowbreak\varphi_{10}=\min\{f_{10}(m):m\in \{4,5,6\}\}$. Since it can be easily seen that $\rho_{\Delta}^s(K_4)=3$, $\rho_{\Delta}^s(K_5)=5$ and $\rho_{\Delta}^s(K_6)=6$, we obtain $f_{10}(4)=f_{10}(6)=\rho_{\Delta}^s(K_4)+\rho_{\Delta}^s(K_6)=3+6=9$ and $f_{10}(5)=2\rho_{\Delta}^s(K_5)=10$. Hence $\tdom(\Johnson{10}{3})\leq \varphi_{10}=9$.

Next, we proceed as follows. First, we establish the value of $\rho_{\Delta}^s(K_n)$ for each $n\geq 4$. Then, we determine $\varphi_n$ for each $n\geq 8$, as well as the values of $m$ that minimize (\ref{eq: f_n(m)}).

\begin{prop}\label{prop:CEOT_complete}
    For $n\geq 6$, we have
    $$\rho_{\Delta}^s(K_n)=\left\lceil\frac{n(n-1)}{5}\right\rceil.$$
\end{prop}
\begin{proof}
    Let $n\geq 6$. Consider the graph $K_n$, and let $D$ be a CEOT of $K_n$ with $|D|=\rho_\Delta^s(K_n)$.
    By Lemma~\ref{lem: edges CEOT}, we have
    $$\rho_{\Delta}^s(K_n)=|D|\geq \frac{2}{5}|E(K_n)|=\frac{n(n-1)}{5}.$$
    Therefore,
    \begin{equation}\label{eq: cota inf CEOT minimo}
        \rho_\Delta^s(K_n) \geq \left\lceil\frac{n(n-1)}{5}\right\rceil.
    \end{equation}

    On the other hand let $D$ be a minimum edge diamond covering of $K_n$. By Theorem~\ref{thm: cov-diam}, $$|D|=\text{cov}(K_n,\text{diam}) = \left\lceil\frac{n(n-1)}{10}\right\rceil.$$ Let $D'$ be the CEOT obtained from $D$ with $|D'|=2|D|=2\left\lceil\frac{n(n-1)}{10}\right\rceil.$ A simple calculation implies that for any $n=5k+\ell$, where $\ell \in \{0,1,3\}$, it holds that $2\left\lceil\frac{n(n-1)}{10}\right\rceil =\left\lceil\frac{n(n-1)}{5}\right\rceil$. Thus, if $n \equiv\, \ell \, (\text{mod }5)$ for $\ell \in \{0,1,3\}$, then $\rho_\Delta^s(K_n) \leq |D'| =\left\lceil\frac{n(n-1)}{5}\right\rceil$.

   Now let $n \in \{5k+2,5k+4\}$ for some $k \in \mathbb{N}$ and let $D$ be a maximum edge diamond packing of $K_n$ and $P$ the set of edges of $K_n$ not covered by $D$. By Theorem~\ref{thm: cov-diam}
    $$|D|=\text{pack}(K_n,\text{diam})=\left\lfloor\frac{n(n-1)}{10}\right\rfloor$$
    and by Observation~\ref{obs: leave packing diamond} $|P|=1$ if $n \notin \{7,9\}$ and $|P|=6$ if $n \in \{7,9\}$.

    First let $n \in \{5k+2,5k+4\}$ for some $k \geq 2$. Then $|P|=1$ and thus $D$ covers all edges of $K_n$ except one, say $ab \in E(K_n)$. Then CEOT $D'$ of $K_n$ can be obtained from $D$ by taking two triangles from each diamond in $D$ and one more triangle $\{a,b,c\}$ for any $c\in V(K_n)\setminus \{a,b\}$. Note that also triangle $\{a,b,c\}$ is overlapping as there exists triangle $T \in D'$ that contains the edge $ac$. Hence
    $$\rho_\Delta^s(K_n) \leq 2|D|+1=2\left\lfloor\frac{n(n-1)}{10}\right\rfloor+1=\left\lceil\frac{n(n-1)}{5}\right\rceil,$$ since $n \in \{5k+2,5k+4:\, k \geq 2\}$.

    For $n=7$, there exists CEOT 
    \begin{equation}\notag
            \begin{aligned}
                D'=\big\{&\{1,2,3\},\{1,2,4\},\{1,5,6\},\{1,5,7\},\{3,4,5\},\\
                &\{3,4,6\},\{2,3,7\},\{4,6,7\},\{2,5,6\}\big\}
            \end{aligned}
        \end{equation}
    of cardinality $|D'|=9=\left\lceil\frac{7\cdot 6}{5}\right\rceil$ and for $n=9$ there exists CEOT
    \begin{equation}\notag
            \begin{aligned}
                D'=\big\{&\{1,2,3\},\{1,2,4\},\{1,5,6\},\{1,5,7\},\{2,6,7\},\{2,6,8\},\{3,4,6\},\{3,6,9\},\\
                &\{3,5,8\},\{3,7,8\},\{2,5,9\},\{4,5,9\},\{1,8,9\},\{4,7,8\},\{7,8,9\}\big\}
            \end{aligned}
        \end{equation}
    of cardinality $|D'|=15=\left\lceil\frac{9\cdot 8}{5}\right\rceil$.

   In any case, it is possible to obtain a CEOT of $K_n$ of cardinality $\left\lceil\frac{n(n-1)}{5}\right\rceil$.
\end{proof}

Proposition~\ref{prop:CEOT_complete} implies that for $n\geq 6$
\begin{equation}
    \rho_{\Delta}^s(K_n)=\left\lceil\frac{n(n-1)}{5}\right\rceil=\frac{n^2-n+a(n)}{5},
\end{equation}
where $a(n)=5\left(\left\lceil\frac{n(n-1)}{5}\right\rceil-\frac{n(n-1)}{5}\right)$. Moreover, for $n\in\{4,5\}$ it is not hard to prove that
\begin{equation}\notag
\begin{aligned}
    \rho_{\Delta}^s(K_4)&=3=\left\lceil\frac{4(4-1)}{5}\right\rceil=\frac{4^2-4+3}{5},\\
    \rho_{\Delta}^s(K_5)&=5=\left\lceil\frac{5(5-1)}{5}\right\rceil+1=\frac{5^2-5+5}{5}.
\end{aligned}
\end{equation}
That is, for $n\geq 4$, we obtain
$$\rho_{\Delta}^s(K_n)=\frac{n^2-n+a(n)}{5},$$
where $$a(n)=\begin{cases}
    0, &\text{if }n\neq 5\text{ and }n\equiv 0,1 \mod 5,\\
    3, &\text{if }n\equiv 2,4 \mod 5,\\
    4, &\text{if }n\equiv 3 \mod 5,\\
    5, &\text{if }n=5.
\end{cases}$$

In the following result, we determine the value of $\varphi_n$ for all $n\geq 8$.

\begin{prop}\label{prop: varphi_n}
    For $n\geq 8$, we have
    $$\varphi_n= \begin{cases}
    \left\lceil\dfrac{n^2-2n}{10}\right\rceil, &\text{if }n\equiv 0,1,2,3,9\mod 10 \text{ and }n\geq 12,\\[10pt]
    \left\lceil\dfrac{n^2-2n}{10}\right\rceil+1, &\text{if }n\equiv 4,5,6,7,8\mod 10 \text{ or }n\in\{9,10,11\}.
\end{cases}$$
\end{prop}

\begin{proof}
Let $n \geq 8$ and let $m\in \left[4..\left\lfloor\frac{n}{2}\right\rfloor\right]$. 
Then
\begin{equation}\label{eq: f_n(m)2}
\begin{aligned}
    f_n(m)&=\frac{m^2-m+a(m)}{5}+\frac{(n-m)^2-(n-m)+a(n-m)}{5}\\
    &=\frac{m^2-m+n^2-2nm+m^2-n+m}{5}+\frac{a(m)+a(n-m)}{5}\\
    &=\frac{n^2-n}{5}+\frac{2}{5}m(m-n)+\frac{a(m)+a(n-m)}{5}.
\end{aligned}
\end{equation}
Note that $m=\left\lfloor\frac{n}{2}\right\rfloor-t$ for some $t\in\left\{0,\ldots,\left\lfloor\frac{n}{2}\right\rfloor-4\right\}$. Then, from (\ref{eq: f_n(m)2}) we have $f_n(m)=f_n\left(\left\lfloor\frac{n}{2}\right\rfloor-t\right)$ where
\begin{equation}\notag
\begin{aligned}
    f_n\left(\left\lfloor\frac{n}{2}\right\rfloor-t\right)&=\frac{n^2-n}{5}+\frac{2}{5}\left(\left\lfloor\frac{n}{2}\right\rfloor-t\right)\left(\left\lfloor\frac{n}{2}\right\rfloor-t-n\right)+\frac{a(\left\lfloor\frac{n}{2}\right\rfloor-t)+a(n-\left\lfloor\frac{n}{2}\right\rfloor+t)}{5}\\
    &=\frac{n^2-n}{5}+\frac{2}{5}\left(\left\lfloor\frac{n}{2}\right\rfloor-t\right)\left(-\left\lceil\frac{n}{2}\right\rceil-t\right)+\frac{a(\left\lfloor\frac{n}{2}\right\rfloor-t)+a(\left\lceil\frac{n}{2}\right\rceil+t)}{5}\\
    &=\frac{n^2-n}{5}-\frac{2}{5}\left\lfloor\frac{n}{2}\right\rfloor\left\lceil\frac{n}{2}\right\rceil + \frac{1}{5}\left(2t\left(t+\left\lceil\frac{n}{2}\right\rceil-\left\lfloor\frac{n}{2}\right\rfloor\right)+a\left(\left\lfloor\frac{n}{2}\right\rfloor-t\right)+a\left(\left\lceil\frac{n}{2}\right\rceil+t\right)\right).
\end{aligned}
\end{equation}
That is, determining $m$ that minimizes $f_n(m)$ is equivalent to determining $t\in\left\{0,\ldots,\left\lfloor\frac{n}{2}\right\rfloor-4\right\}$ that minimizes
\begin{equation}
    g_n(t)=2t\left(t+\left\lceil\frac{n}{2}\right\rceil-\left\lfloor\frac{n}{2}\right\rfloor\right)+a\left(\left\lfloor\frac{n}{2}\right\rfloor-t\right)+a\left(\left\lceil\frac{n}{2}\right\rceil+t\right)
\end{equation}

We will analyze $g_n(t)$ separately for $n$ even and $n$ odd.
\begin{itemize}
    \item If $n$ is even, then
    $$g_n(t)=2t^2+a\left(\frac{n}{2}-t\right)+a\left(\frac{n}{2}+t\right).$$
    Note that on one hand $g_n(0)=2\,a\left(\frac{n}{2}\right)\leq 10$, while for $t\geq 3$, we have $g_n(t)\geq 2t^2\geq 18$.
    Therefore, if $t$ minimizes $g_n(t)$, then $t\in\{0,1,2\}$.
    
    \item If $n$ is odd, then
    $$g_n(t)=2t(t+1)+a\left(\frac{n-1}{2}-t\right)+a\left(\frac{n+1}{2}+t\right).$$
    Note that $g_n(0)=a\left(\frac{n-1}{2}\right)+a\left(\frac{n+1}{2}\right)\leq 8$, while for $t\geq 2$, we have $g_n(t)\geq 2t(t+1)\geq 12$.
    Therefore, if $t$ minimizes $g_n(t)$, then $t\in\{0,1\}$.
\end{itemize}

In Table~\ref{tab: n chicos} we show, for $n\in[8..14]$, the computed values of $g_n(t)$ with $t\in\{0,1,2\}$ if $n$ is even, and $t\in\{0,1\}$ if $n$ is odd. The minimum of these values is highlighted in bold. These cases are analyzed separately because, in computing $g_n(t)$ for some of the values of $t$ considered, $a(5)$ appears, which follows a different rule from the general one.

\renewcommand{\arraystretch}{1.2}
\begin{table}[ht]
    \centering
    \begin{tabular}{|c|c|c|c|c|c|c|c|}
        \hline
        $n$ & $8$ & $9$ & $10$ & $11$ & $12$ & $13$ & $14$ \\ \hline
        $a\left(\left\lfloor\frac{n}{2}\right\rfloor\right)$
        & $3$ & $3$ & $5$ & $5$ & $0$ & $0$ & $3$ \\ \hline
        $a\left(\left\lceil\frac{n}{2}\right\rceil\right)$
        & $3$ & $5$ & $5$ & $0$ & $0$ & $3$ & $3$ \\ \hline

        \rowcolor{gray!20}
        $g_n(0)$
        & $\mathbf 6$ & $\mathbf 8$ & $10$ & $\mathbf 5$ & $\mathbf 0$ & $\mathbf 3$ & $\mathbf 6$ \\ \hline

        $a\left(\left\lfloor\frac{n}{2}\right\rfloor-1\right)$
        & $-$ & $-$ & $3$ & $3$ & $5$ & $5$ & $0$ \\ \hline
        $a\left(\left\lceil\frac{n}{2}\right\rceil+1\right)$
        & $-$ & $-$ & $0$ & $3$ & $3$ & $4$ & $4$ \\ \hline

        \rowcolor{gray!20}
        $g_n(1)$
        & $-$ & $-$ & $\mathbf 5$ & $10$ & $10$ & $13$ & $\mathbf 6$ \\ \hline

        $a\left(\left\lfloor\frac{n}{2}\right\rfloor-2\right)$
        & $-$ & $-$ & $-$ & $-$ & $3$ & $-$ & $5$ \\ \hline
        $a\left(\left\lceil\frac{n}{2}\right\rceil+2\right)$
        & $-$ & $-$ & $-$ & $-$ & $4$ & $-$ & $3$ \\ \hline

        \rowcolor{gray!20}
        $g_n(2)$
        & $-$ & $-$ & $-$ & $-$ & $15$ & $-$ & $16$ \\ \hline
    \end{tabular}
    \caption{Values of $g_n(t)$ for $n\in[8..14]$.}
    \label{tab: n chicos}
\end{table}
\renewcommand{\arraystretch}{1}

In Table~\ref{tab: n general} we present, for $n\geq 15$, the values of $g_n(t)$ with $t\in\{0,1,2\}$ if $n$ is even, and $t\in\{0,1\}$ if $n$ is odd. When computing $g_n(t)$, since the general rule for $a(n)$ depends on $n$ modulo $5$, the computation of $a\left(\left\lfloor\frac{n}{2}\right\rfloor-t\right)+a\left(\left\lceil\frac{n}{2}\right\rceil+t\right)$ depends on $n$ modulo $10$. Each column corresponds to a value $n\equiv k \mod 10$, with $k\in\{0,\dots,9\}$. Again, the minimum of these values is highlighted in bold.

\begin{table}[ht]
    \centering
    \renewcommand{\arraystretch}{1.2}
    \begin{tabular}{|c|c|c|c|c|c|c|c|c|c|c|}
    \hline
        $k$ & $0$ & $1$ & $2$ & $3$ & $4$ & $5$ & $6$ & $7$ & $8$ & $9$ \\ \hline
        $a\left(\left\lfloor\frac{n}{2}\right\rfloor\right)$
        & $0$ & $0$ & $0$ & $0$ & $3$ & $3$ & $4$ & $4$ & $3$ & $3$ \\ \hline
        $a\left(\left\lceil\frac{n}{2}\right\rceil\right)$
        & $0$ & $0$ & $0$ & $3$ & $3$ & $4$ & $4$ & $3$ & $3$ & $0$ \\ \hline

        \rowcolor{gray!20}
        $g_n(0)$
        & $\mathbf 0$ & $\mathbf 0$ & $\mathbf 0$ & $\mathbf 3$ & $\mathbf 6$ & $\mathbf 7$ & $\mathbf 8$ & $\mathbf 7$ & $\mathbf 6$ & $\mathbf 3$ \\ \hline

        $a\left(\left\lfloor\frac{n}{2}\right\rfloor-1\right)$
        & $3$ & $3$ & $0$ & $0$ & $0$ & $0$ & $3$ & $3$ & $4$ & $4$ \\ \hline
        $a\left(\left\lceil\frac{n}{2}\right\rceil+1\right)$
        & $0$ & $3$ & $3$ & $4$ & $4$ & $3$ & $3$ & $0$ & $0$ & $0$ \\ \hline

        \rowcolor{gray!20}
        $g_n(1)$
        & $5$ & $10$ & $5$ & $8$ & $\mathbf 6$ & $\mathbf 7$ & $\mathbf 8$ & $\mathbf 7$ & $\mathbf 6$ & $8$ \\ \hline

        $a\left(\left\lfloor\frac{n}{2}\right\rfloor-2\right)$
        & $4$ & $-$ & $3$ & $-$ & $0$ & $-$ & $0$ & $-$ & $3$ & $-$ \\ \hline
        $a\left(\left\lceil\frac{n}{2}\right\rceil+2\right)$
        & $3$ & $-$ & $4$ & $-$ & $3$ & $-$ & $0$ & $-$ & $0$ & $-$ \\ \hline

        \rowcolor{gray!20}
        $g_n(2)$
        & $15$ & $-$ & $15$ & $-$ & $11$ & $-$ & $\mathbf 8$ & $-$ & $11$ & $-$ \\ \hline
    \end{tabular}
    \renewcommand{\arraystretch}{1}
    \caption{Values of $g_n(t)$ for $n\geq 15$, each column corresponds to $n\equiv k \mod 10$.}
    \label{tab: n general}
\end{table}

Note that, for all $n\neq 10$, the minimum is attained at $t=0$. That is, 
$$\varphi_n=f_n\left(\left\lfloor\frac{n}{2}\right\rfloor\right)=\frac{n^2-n}{5}-\frac{2}{5}\left\lfloor\frac{n}{2}\right\rfloor\left\lceil\frac{n}{2}\right\rceil + \frac{1}{5}\underbrace{\left(a\left(\left\lfloor\frac{n}{2}\right\rfloor\right)+a\left(\left\lceil\frac{n}{2}\right\rceil\right)\right)}_{=g_n(0)}$$

For even $n\neq 10$, we have
$$\varphi_n=\frac{n^2-n}{5}-\frac{2}{5}\left(\frac{n}{2}\right)^2+\frac{1}{5}g_n(0)=\frac{n^2-2n+2g_n(0)}{10}.$$
For $n$ odd, we have
$$\varphi_n=\frac{n^2-n}{5}-\frac{2}{5}\frac{n^2-1}{4}+\frac{1}{5}g_n(0)=\frac{n^2-2n+1+2g_n(0)}{10}.$$
Finally, for $n=10$, we have $\varphi_{10}=f_{10}(4)=9=\frac{10^2-2\cdot 10+10}{10}$. Thus for all $n\geq 8$, we obtain
\begin{equation}\label{eq: varphi cases}
    \varphi_n=\begin{cases}
    \dfrac{n^2-2n}{10}, &\text{if }n\equiv 0,2\mod 10\quad (n\neq 10),\\
    \dfrac{n^2-2n+1}{10}, &\text{if }n\equiv 1\mod 10 \quad (n\neq 11),\\
    \dfrac{n^2-2n+7}{10}, &\text{if }n\equiv 3,9\mod 10 \quad (n\neq 9),\\
    \dfrac{n^2-2n+10}{10}, &\text{if }n=10,\\
    \dfrac{n^2-2n+11}{10}, &\text{if }n=11,\\
    \dfrac{n^2-2n+12}{10}, &\text{if }n\equiv 4,8\mod 10,\\
    \dfrac{n^2-2n+15}{10}, &\text{if }n\equiv 5,7\mod 10,\\
    \dfrac{n^2-2n+16}{10}, &\text{if }n\equiv 6\mod 10,\\
    \dfrac{n^2-2n+17}{10}, &\text{if }n=9.\\
\end{cases}
\end{equation}

Equivalently, $\varphi_n$ is given by
$$\varphi_n=\begin{cases}
    \left\lceil\dfrac{n^2-2n}{10}\right\rceil, &\text{if }n\equiv 0,1,2,3,9\mod 10 \text{ and }n\geq 12,\\[10pt]
    \left\lceil\dfrac{n^2-2n}{10}\right\rceil+1, &\text{if }n\equiv 4,5,6,7,8\mod 10 \text{ or }n\in\{9,10,11\}.
\end{cases}$$
\end{proof}

\begin{ex}
    Consider $n=16$. From what was done in the proof of Proposition~\ref{prop: varphi_n}, we have $\varphi_{16}=24$. Besides, $\varphi_{16}=24=f_{16}(6)=f_{16}(7)=f_{16}(8)$.
    This is, $\varphi_{16}$ is the cardinality of a minimum CEOT of the graphs $K_6+K_{10}$, $K_7+K_9$, and $K_8+K_8$.
\end{ex}

The proof of Proposition~\ref{prop: varphi_n} also determines the values of $t$ for which $f_n\left(\left\lfloor\frac{n}{2}\right\rfloor-t\right)$ attains its minimum, which is precisely $\varphi_n$. These values of $t$ are given in the following result.

\begin{remark}\label{obs: m-min-f_n(m)}
    For each $n\geq 8$, we have $\varphi_n=f_n\left(\left\lfloor\frac{n}{2}\right\rfloor-t\right)$ where $t$ is given by
    \begin{itemize}
        \item $t\in\{0,1,2\}$ for $n\equiv 6 \mod 10$,
        \item $t\in\{0,1\}$ for $n\equiv 4,5,7,8 \mod 10$ ($n\neq 8$),
        \item $t=0$ for $n\equiv 0,1,2,3,9 \mod 10$ or $n=8$,
        \item $t=1$ for $n=10$.
    \end{itemize}
\end{remark}

We have already seen that $\tdom(\Johnson{n}{3})\leq \varphi_n$ for all $n\geq 8$ and we have determined the value of $\varphi_n$. In the remainder of this subsection, we will prove that for $n\geq 10$ this bound is tight, i.e. $\tdom(\Johnson{n}{3})=\varphi_n$. Before that, we need the following.

Mantel's theorem states that a triangle-free graph has at most $\left\lfloor \frac{n^2}{4}\right\rfloor$ edges~\cite{mantel1907}. Since bipartite graph is triangle-free, it contains at most  $\left\lfloor \frac{n^2}{4}\right\rfloor$ edges. For non-bipartite graphs Erd\H{o}s~\cite{bondy2008} improved Mantel's theorem by showing the following.

\begin{lemma}
    If $G$ is not bipartite and triangle-free, then
    \[
    |E(G)|\leq \frac{(n-1)^2}{4}+1.
    \]
    \label{lem:edges-trianglefree}
\end{lemma}

\begin{lemma}\label{lema: dom total vert-disj}
     If $n\geq 10$ and $D$ is a minimum total dominating set of $\Johnson{n}{3}$, then $V(G^D)$ can be partitioned into two cliques.
\end{lemma}
\begin{proof}
    Let $n\geq 10$, let $D$ be a minimum total dominating set of $\Johnson{n}{3}$, and let $G=G^D$. From the definition of $G$, it follows that $G$ is a graph of order $n$ with $\alpha(G)\leq 2$, and that $D$ is a CEOT of $G$.

    Suppose $V(G)$ does not admit a partition into two cliques. Consequently, the graph $\overline{G}$ is not bipartite. Also, $\omega(\overline{G})=\alpha(G)\leq 2$, so $\overline{G}$ is triangle-free. By Lemma~\ref{lem:edges-trianglefree}, we obtain
    $$|E(\overline{G})|\leq \frac{(n-1)^2}{4}+1,$$
    and we have
    $$|E(G)|=\binom{n}{2}-|E(\overline{G})|\geq \frac{n(n-1)}{2}-\frac{(n-1)^2}{4}-1=\frac{n^2-5}{4}.$$
    On the other hand, by Lemma~\ref{lem: edges CEOT} and Proposition~\ref{prop: cota sup tdom varphi_n} it follows that
    $$|E(G)|\leq \frac{5}{2}\,|D|=\frac{5}{2}\,\tdom(\Johnson{n}{3})\leq \frac{5}{2}\,\varphi_n.$$
    Note that, from (\ref{eq: varphi cases}), we have 
    $$\varphi_n=\frac{n^2-2n+c(n)}{10},$$
    where $c(n)$ is a constant depending on $n$ with $c(n)\in\{0,1,7,10,11,12,15,16,17\}$.
    Note that
    \begin{equation}\notag
    \begin{aligned}
        \frac{n^2-5}{4}-\frac{5}{2}\,\varphi_n&=\frac{n^2-5}{4}-\frac{5}{2}\left(\frac{n^2-2n+c(n)}{10}\right)\\
        &=\frac{n^2-5}{4}-\frac{n^2-2n+c(n)}{4}=\frac{2n-5-c(n)}{4}.
    \end{aligned}
    \end{equation}
    Since from the values of $n$ and $c(n)$ we can easily deduce that $2n-5-c(n) > 0$ for any $n \geq 10$, we get
    $$\frac{n^2-5}{4}\leq |E(G)|\leq \frac{5}{2}\,\varphi_n <\frac{n^2-5}{4},$$
    and we arise to a contradiction. Therefore, we conclude that $V(G)$ admits a partition into two cliques.
\end{proof}

We are now in a position to prove the following result.
\begin{theorem}\label{teo: tdom J(n,3)}
    For $n\geq 10$, it holds $\tdom(\Johnson{n}{3})=\varphi_n$.
\end{theorem}
\begin{proof}
    By Proposition~\ref{prop: cota sup tdom varphi_n} we know that $\tdom(J(n,3)) \leq \varphi_n$. To show the reversed inequality, let $n\geq 10$, let $D$ be a minimum total dominating set of $\Johnson{n}{3}$ and let $G=G^D$. By Lemma~\ref{lema: dom total vert-disj} it follows that $V(G)$ can be partitioned into two cliques $A$ and $B$, with $|A|\leq |B|$, $|A|=m$, and $|B|=n-m$. 
    
    Let $E(A,B)$ be the set of edges with one endpoint in $A$ and one endpoint in $B$. 
    If $E(A,B)=\emptyset$, then $G$ is a disjoint union of complete graphs and we obtain $|D|=\varphi_n$. Therefore, suppose $E(A,B)\neq \emptyset$. In this case, if $e=ab\in E(A,B)$ with $a\in A$ and $b\in B$, then since $D$ is an edge triangle covering, there exists $c\in [n]\setminus\{a,b\}$ such that $\{a,b,c\}\in D$. If $c\in A$, then $bc\in E(A,B)$ and if $c\in B$, then $ac\in E(A,B)$. In either case, we have $|E(A,B)|\geq 2$.

    Note that we may assume $m\neq 2$. Indeed, if $A=\{x_1,x_2\}$, since $x_1x_2\in E(G)$ and $D$ is a CEOT, there exists a triangle $T=\{x_1,x_2,b\}\in D$ for some $b\in B$. Then, considering $A'=A\cup\{b\}$ and $B'=B\setminus \{b\}$, we obtain a new partition of $V(G)$ into two cliques $A'$ and $B'$ with $|A'|= 3$ and $|B'|=n-3\geq 7$. Therefore, the partition $(A,B)$ of $V(G)$ can be chosen so that $m \neq 2$.
    
    Let us now see that $m\neq 3$.
    Note that
    \begin{equation}\notag
        \begin{aligned}
            |E(G)|&=\binom{m}{2}+\binom{n-m}{2}+|E(A,B)|\\
            &=\frac{m(m-1)}{2}+\frac{(n-m)(n-m-1)}{2}+|E(A,B)|\\
            &=\frac{n^2-n+2m^2-2nm}{2}+|E(A,B)|=\frac{n^2-n}{2}+m(m-n)+|E(A,B)|.
        \end{aligned}
    \end{equation}
    Moreover, by Lemma~\ref{lem: edges CEOT}, it follows that
    \begin{equation}\label{eq: cota sup D}
        |D|\geq \frac{2}{5}|E(G)|=\frac{n^2-n}{5}+\frac{2}{5}m(m-n)+\frac{2}{5}|E(A,B)|
    \end{equation}
    On the other hand, let $m'=\left\lfloor\frac{n}{2}\right\rfloor$ for $n\geq 11$, and $m'=4$ for $n=10$. By Remark~\ref{obs: m-min-f_n(m)} we have $\varphi_n=f_n\left(m'\right)$, where
    $$f_n(m')=\frac{n^2-n}{5}+\frac{2}{5}m'(m'-n)+\frac{a(m')+a(n-m')}{5}.$$
    Observe that if $n=10$, then $a(m')+a(n-m')=a(4)+a(6)=3+0=3\leq 8$. If $n=11$, then $a(m')+a(n-m')=a(5)+a(6)=5+0=5\leq 8$. Finally, if $n\geq 12$, then $\left\lfloor\frac{n}{2}\right\rfloor\geq 6$ and $n-\left\lfloor\frac{n}{2}\right\rfloor\geq 6$. Therefore, $a(m')+a(n-m')\leq 4+4=8$. That is, for every $n\geq 10$ we have $a(m')+a(n-m')\leq 8$.
    
    By Proposition~\ref{prop: cota sup tdom varphi_n} we have $|D|=\tdom(\Johnson{n}{3})\leq\varphi_n$, then for every $n\geq 10$, we have 
    \begin{equation}\label{eq: cota inf D}
        |D|\leq\varphi_n=f_n(m')\leq \frac{n^2-n}{5}+\frac{2}{5}m'(m'-n)+\frac{8}{5}.
    \end{equation}
    From (\ref{eq: cota sup D}) and (\ref{eq: cota inf D}), it follows that
    \begin{equation}\label{eq: comp m y m'}
        \begin{aligned}
            \frac{n^2-n}{5}+\frac{2}{5}m(m-n)+\frac{2}{5}|E(A,B)|&\leq \frac{n^2-n}{5}+\frac{2}{5}m'(m'-n)+\frac{8}{5}\\
            m(m-n)+|E(A,B)|&\leq m'(m'-n)+4.
        \end{aligned}
    \end{equation}
    If $m=3$, since $|E(A,B)|\geq 2$, then
    $$3(3-n)+2\leq 3(3-n)+|E(A,B)|\leq m'(m'-n)+4,$$
    so
    \begin{equation}\label{eq: cond m'}
        m'(m'-n)-3(3-n)+2\geq 0.
    \end{equation}
    If $n=10$, then $m'=4$ and
    $$m'(m'-n)-3(3-n)+2=4(-6)-3(-7)+2=-24+21+2=-1<0.$$
    If $n\geq 11$, then $m'=\left\lfloor\frac{n}{2}\right\rfloor$, so
    \begin{equation}\notag
    \begin{aligned}
        &m'(m'-n)-3(3-n)+2=-\left\lfloor\frac{n}{2}\right\rfloor\left\lceil\frac{n}{2}\right\rceil-3(3-n)+2\leq \frac{-n^2+1}{4}-3(3-n)+2\\
        &=-\frac{1}{4}\left(n^2-1+12(3-n)-8\right)=-\frac{1}{4}\left(n^2-12n+27\right)=-\frac{1}{4}\underbrace{(n-3)}_{>0}\underbrace{(n-9)}_{>0}<0.
    \end{aligned}
    \end{equation}
    In either case, we obtain $m'(m'-n)-3(3-n)+2<0$, contradicting (\ref{eq: cond m'}). This contradiction arises from assuming $m=3$. Therefore, we have $m\notin\{2,3\}$.

    Let $T_1,\ldots, T_k$ be the triangles of $D$ having nonempty intersection with both $A$ and $B$, ordered so that for some $\ell\in\{0,1,\ldots,k\}$ we have $|T_i\cap A|=2$ if and only if $i\leq \ell$. 

    Let $D_0:=D$. We will iterate until obtaining a collection of triangles $D'$ with $|D'|\leq |D|$ such that the graph $G^{D'}$ is a disjoint union of two complete graphs. To do this, at step $i$, we will obtain a collection of triangles $D_{i}$ of $G$ such that
    \begin{enumerate}
        \item $D_i\cap \{T_1,\ldots, T_i\}=\emptyset$,
        \item $\{T_{i+1},\dots, T_k\}\subseteq D_i$, and these are the only triangles of $D_i$ having nonempty intersection with both $A$ and $B$,
        \item $|D_i|\leq |D|$, 
        \item $D_i$ covers all edges of $E(G)\setminus E(A,B)$, and
        \item\label{eq: condicion D_{i-1}} for each $T\in D_i$ such that $T\subseteq A$ or $T\subseteq B$, there exists $T'\in D_i$ such that $|T\cap T'|=2$.
    \end{enumerate}

    At step $i$ we proceed as follows. Let $D_{i-1}$ be the collection of triangles obtained at step $(i-1)$, which satisfies these 5 conditions.
    
    If $i\leq \ell$, then $T_i=\{a_1^i,a_2^i,b^i\}$ for some $a_1^i,a_2^i\in A$ and $b^i\in B$. We define $D_i$ as follows.
    \begin{itemize}
        \item If $|\{a_1^i,a_2^i\}\cap T|\leq 1$ for every $T\in D_{i-1}$ with $T\subseteq A$, then we choose $a\in A\setminus\{a_1^i,a_2^i\}$ (which exists since $|A|\neq 2$). We define $T_i'=\{a_1^i,a_2^i,a\}$ and $D_i=(D_{i-1}\setminus\{T_i\})\cup \{T_i'\}$.

        \item If there is a unique $T_0\in D_{i-1}$ with $\{a_1^i,a_2^i\}\subseteq T_0\subseteq A$, then we choose $a\in A\setminus T_0$ (which exists since $|A|\neq 3$). We again define $T_i'=\{a_1^i,a_2^i,a\}$ and $D_i=(D_{i-1}\setminus\{T_i\})\cup \{T_i'\}$.

        \item If there exist two triangles $T_0,S_0\in D_{i-1}$ such that $\{a_1^i,a_2^i\}\subseteq T_0\subseteq A$ and $\{a_1^i,a_2^i\}\subseteq S_0\subseteq A$, then we define $D_i=D_{i-1}\setminus\{T_i\}$.
    \end{itemize}
    Observe that the triangles $T_i'$ defined in the first two cases do not belong to $D_{i-1}$.

    Note that the only edge of $E(G)\setminus E(A,B)$ covered by $T_i$ is $a_1^ia_2^i$. In the first two cases, we see that this edge is also covered by $T_i'$, while in the last case, there exist at least two triangles in $D_{i-1}\setminus\{T_i\} \subseteq D_i$ that cover it. In either case, $D_i$ covers all edges of $E(G)\setminus E(A,B)$. 
    Moreover, the only triangles in $D_i$ having nonempty intersection with both $A$ and $B$ are the triangles $T_j$ for $j>i$.

    Finally, let $T\in D_i$ be such that $T\subseteq A$ or $T\subseteq B$. Let us see that there exists $T'\in D_i$ such that $|T\cap T'|=2$.
    First, if $T\in D_{i-1}\setminus\{T_i\}$, since (\ref{eq: condicion D_{i-1}}) holds we know there exists $T'\in D_{i-1}$ such that $|T\cap T'|=2$. If $T'\neq T_i$, then $T'\in D_i$, and we have found $T'\in D_i$ such that $|T\cap T'|=2$. If $T'=T_i$, note that since $T\subseteq A$ or $T\subseteq B$, we must have $T\cap T'=\{a_1^i,a_2^i\}$. If $T'$ is the only triangle in $D_{i-1}$ such that $\{a_1^i,a_2^i\}\subseteq T'\subseteq A$, then by how we defined $D_i$, $T_i'\in D_i$ also satisfies $|T\cap T_i'|=2$. Otherwise, that is, if there is another triangle $T''\in D_{i-1}$ such that $\{a_1^i,a_2^i\}\subseteq T''\subseteq A$, then $T''\in D_i=D_{i-1}\setminus\{T_i\}$, and $|T\cap T''|=|\{a_1^i,a_2^i\}|=2$ holds.
    If $T=T_i'=\{a_1^i,a_2^i,a\}$, then since edge $a_1^ia$ is covered by some triangle $T'\in D_{i-1}\setminus\{T_i\}\subseteq D_i$, we find $T'\in D_i$ such that $|T\cap T'|=|\{a_1^i,a\}|=2$.

    That is, we obtain a collection of triangles $D_i$ satisfying the 5 mentioned conditions.
    If $i>\ell$ then $T_i=\{a^i,b_1^i,b_2^i\}$ for some $a^i\in A$ and $b_1^i,b_2^i\in B$. We proceed analogously to what was done for $i\leq \ell$ and obtain a collection of triangles $D_i$ that also satisfies the 5 conditions.

    Let now $D'=D_k$. We have $D'\cap \{T_1,\ldots, T_k\}=\emptyset$ and for each $T\in D'$, either $T\subseteq A$ or $T\subseteq B$. Moreover $D'$ covers all edges in $E(G)\setminus E(A,B)$. That is, $D'$ is an edge triangle covering of the graph $G'$, the disjoint union of the complete graphs $G[A]$ and $G[B]$. Moreover, since condition (\ref{eq: condicion D_{i-1}}) holds, $D'$ is a CEOT of $G'$. As $\alpha(G')=2$, it follows from Proposition~\ref{prop:covering-overlapping} that $D'$ is a total dominating set of $J(n,3)$ and thus $\tdom(J(n,3)) \leq |D'| \leq |D|= \tdom(J(n,3))$. Since the definition of $\varphi_n$ implies that $$\varphi_n \leq \rho_{\Delta}^s(K_{|A|})+\rho_{\Delta}^s(K_{|B|}) \leq |D'|=\tdom(J(n,3)),$$ we get $\tdom(J(n,3)) \geq \varphi_n$ as desired.
\end{proof}

\begin{remark}
    In Theorem \ref{teo: tdom J(n,3)} we proved that $\tdom(\Johnson{n}{3})=\varphi_n$ for all $n\geq 10$. We already know that $\tdom(\Johnson{6}{3})=4$ and $\tdom(\Johnson{7}{3})=5$. For $n\in \{8,9\}$, we computed the value of $\tdom(\Johnson{n}{3})$ using CPLEX with integer linear programming, obtaining $\tdom(\Johnson{n}{3})=\varphi_n$ for each $n\geq 6$. 
\end{remark}

\section{Further remarks}\label{sec:FurtherRem}

This work provides a first general study on domination and total domination in Johnson graphs. For $k\in \{2,3\}$, we establish the exact value of $\gamma_t(\Johnson{n}{k})$ (with $\gamma(\Johnson{n}{k})$ following directly from lotto designs), while for the general case $n \geq 2k$, we present bounds and monotonicity results for both parameters. In particular, one of our main results shows the strict increase of $\gamma(\Johnson{n}{k})$ in steps of two in $n$ (Theorem~\ref{thm:strict - 2steps}). As future work, this result motivates further investigation into the analogous behavior of $\gamma_t(\Johnson{n}{k})$.

Regarding the contrast between $\dom(\Johnson{n}{k})$ and $\tdom(\Johnson{n}{k})$, an interesting question is whether either of the known general bounds $\dom(G)\leq \tdom(G)\leq 2\dom(G)$ are tight for the graphs $\Johnson{n}{k}$. As we have obtain the exact values of these parameters for $k\in \{2,3\}$, it follows that neither of these bounds is close-fitting in general, and that both the difference between $\tdom(\Johnson{n}{k})$ and $\dom(\Johnson{n}{k})$ and the difference between $2\dom(\Johnson{n}{k})$ and $\tdom(\Johnson{n}{k})$ grow without bound with $n$ ($k\in\{2,3\}$).

Finally, to advance the study of the (total) dominating sets of $\Johnson{n}{k}$ for $k\geq 4$, we generalize the translation developed for $k=3$ to the setting of $(k-1)$-uniform hypergraphs. This yields a characterization of the (total) dominating sets of $\Johnson{n}{k}$ in terms of coverings of the complete hypergraph.

The translation of (total) domination in $\Johnson{n}{3}$ to edge (overlapping) triangle covering in a graph $G$ can be extended to (total) domination in $\Johnson{n}{k}$ and (overlapping) coverings in $(k-1)$-uniform hypergraphs. For this we need the following. Given a hypergraph $H$, $\alpha(H)$ is the maximum cardinality of a vertex subset of $V(H)$ which does not contain any hyperedge. For a $k$-uniform hypergraph $H$, a subset $K \subseteq [n]$ of cardinality at least $k$ induces a complete subhypergraph of $H$, if every $k$-subset of $K$ is an hyperedge in $H$.

\begin{prop}\label{prop:hypergraph}
    Let $D \subseteq V(\Johnson{n}{k})$ and let $\mathcal{F}$ be the family of all $(k-1)$-uniform hypergraphs $G$ of order $n$ with $\alpha(G) \leq k-1$. Then $D$ is a dominating set of $\Johnson{n}{k}$ if and only if $D$ is a covering of hyperedges of a hypergraph $H$ by complete hypergraphs of order $k$ of some $H\in \mathcal{F}$.
\end{prop}
\begin{proof}
Let $H$ be an $(k-1)$-uniform hypergraph of order $n$ with $V(H)=[n]$. For a family $D$ of $k$-element subsets of $[n]$ we define a hypergraph $H^D$ as follows: $V(G^D)=[n]$ and for each $u \in D$ every $(k-1)$-subset of $u$ is an hyperedge of $H^D$, i.e.\ $u$ induces a complete subhypergraph in $H^D$. Note that $H^D$ is $(k-1)$-uniform hypergraph. 

Let $D$ be a dominating set of $\Johnson{n}{k}$ and let $\{x_1,\ldots ,x_{k-1}\}$ be an arbitrary hyperedge of $H^D$. Definition of $H^D$ implies that there exists $a \in [n]$ such that $x=\{x_1,\ldots ,x_{k-1},a\} \in D$. Hence, any $(k-1)$-subset of $x$ is an hyperedge of $G^D$ and thus the set $\{x_1,\ldots ,x_{k-1},a\}$ induces a complete subhypergraph of $H^D$. Therefore, every hyperedge of $H^D$ is covered with a complete subhypergraph of order $k$. It is also clear that $\alpha(H^D) \leq k-1$. (Note that a set of vertices $A$ of a hypergraph is an independent set if no hyperedge of the hypergraph is contained in $A$). Indeed, if there exists an independent set $A=\{x_1,\ldots ,x_k\}$ of cardinality $k$, then no hyperedge of $H^D$ is contained in $A$. This implies that for any $z=\{z_1,\ldots ,z_{k-1}\} \subseteq A$, it holds that $z \notin E(H^D)$. Hence, vertex $u=A \in V(\Johnson{n}{k})$ is not dominated by $D$, a contradiction.

Conversely, let $H$ be an $(k-1)$-uniform hypergraph of order $n$ with $\alpha(H)\leq k-1$, where each edge of $H$ is contained in a complete subhypergraph of $H$ of order $k$. Let $D'$ be a covering of hyperedges of $H$ with complete subhypergraphs of $H$ of order $k$. Let $u=\{u_1,\ldots ,u_k\} \in V(\Johnson{n}{k})\setminus D'$. Since $\alpha(H) \leq k-1$, $u$ is not an independent set of $H$. Hence, there exists an hyperedge $e$ of $H$ that is contained in $u$. Note that $|e|=k-1$. Since $D'$ is a covering of hyperedges of $H$, there exists $v \in D'$ such that $e \subseteq v$. Hence $u=v$ or $uv \in E(\Johnson{n}{k})$, and consequently $D'$ is a dominating set of $\Johnson{n}{k}$. 
\end{proof}

Notice that an analogous result is true for total domination if we consider overlapping coverings instead of just coverings, where an overlapping covering of hyperedges by complete hypergraph of $G$ is a covering $D$ such that for any $K\in D$ exists $K'\in D$ with $|K\cap K'|=k-1$.

We conclude the paper with the following natural open problem.

\begin{problem}
    Determine $\gamma(J(n,k))$ and $\tdom(J(n,k))$ for $k \geq 4$. In particular, find explicit formulas for these parameters in terms of $n$ and $k$, generalizing the results obtained for $k \in \{2,3\}$.
\end{problem}

\bibliographystyle{acm}
\bibliography{TDom-Johnson/biblio-tdom}

\end{document}